\documentclass[11pt,reqno]{amsart}
\usepackage[margin=1.08in]{geometry}
\usepackage{amsmath,amssymb,amsthm,mathtools,mathrsfs}
\usepackage{microtype,aliascnt}
\usepackage[hidelinks]{hyperref}
\allowdisplaybreaks
\numberwithin{equation}{section}
\newtheorem{theorem}{Theorem}[section]
\newaliascnt{proposition}{theorem}
\newtheorem{proposition}[proposition]{Proposition}
\aliascntresetthe{proposition}
\newaliascnt{lemma}{theorem}
\newtheorem{lemma}[lemma]{Lemma}
\aliascntresetthe{lemma}
\usepackage[nameinlink,noabbrev]{cleveref}
\crefname{theorem}{Theorem}{Theorems}
\crefname{proposition}{Proposition}{Propositions}
\crefname{lemma}{Lemma}{Lemmas}
\newcommand{\T}{\mathbb T}
\newcommand{\R}{\mathbb R}
\newcommand{\Z}{\mathbb Z}
\newcommand{\Q}{\mathbb Q}
\newcommand{\1}{\mathbf 1}
\newcommand{\dist}{\operatorname{dist}}

\newcommand{\vol}{\operatorname{vol}}

\hypersetup{%
  pdftitle={Maximal volume of convex bodies avoiding random points and lacunary dilates},
  pdfauthor={Yuval Peres and Bohan Yang}
}

\title[Maximal volume avoiding lacunary dilates]{Maximal volume of convex bodies\\avoiding random points and lacunary dilates}
\author[Yuval Peres and Bohan Yang]{Yuval Peres$^{1}$ and Bohan Yang$^{2,3}$}
\makeatletter
\let\original@setauthors\@setauthors
\def\@setauthors{\original@setauthors
  \begin{center}\small\itshape
  $^{1}$ Beijing Institute of Mathematical Sciences and Applications (BIMSA),\\
  Beijing 101408, China\\[3pt]
  $^{2}$ Shanghai Institute for Mathematics and Interdisciplinary Sciences (SIMIS),\\
  Shanghai 200433, China\\[3pt]
  $^{3}$ Research Institute of Intelligent Complex Systems, Fudan University,\\
  Shanghai 200433, China
  \end{center}}
\makeatother

\date{}

\begin{document}

\begin{abstract}
For a positive real lacunary sequence and any probability measure $\mu$ with polynomial Fourier decay, we determine, for $\mu$-almost every $x$, the asymptotic maximal volume of convex bodies avoiding the first $N$ dilates of $x$ modulo $\Z^d$.  Among homothets of a fixed convex body, the maximal volume is asymptotic to $(\log N)/N$. When all convex bodies in the unit cube are allowed, it is asymptotic to $d(\log N)/N$. For independent uniform points in a fixed convex body $\Omega\subset\R^d$, the maximal volume of an empty convex body is asymptotic to $d\vol(\Omega)(\log N)/N$ almost surely.
We also show, by constructing a one-dimensional counterexample, that lacunarity cannot be replaced by sparsity on sublacunary scales.
\end{abstract}
\maketitle

\section{Introduction}
Let $d\ge1$ and let $(a_n)$ be a sequence of positive real numbers.
For $x\in\R^d$, consider
$$
 A_N(x)=\{a_nx\bmod1:1\le n\le N\},
$$
where $y\bmod1$ denotes the coordinatewise fractional part of $y$, taking
values in $[0,1)^d$.
A convex body is called empty relative to a point set if its interior
contains no point of that set. We study the maximal volume of empty
convex bodies relative to $A_N(x)$, first among homothets of a fixed
convex body and then among all convex bodies in the unit cube.

\subsection{Main results}
We identify $[0,1)^d$ with $\T^d=\R^d/\Z^d$.  Write $\vol(K)$ for
Lebesgue volume and $\|v\|$ for the Euclidean norm; write
$\|v\|_\infty=\max_j|v_j|$ for the maximum norm.
For the results on lacunary dilates, assume that, for some $q>1$,
\begin{equation}\label{eq:lacunarity}
 a_{n+1}\ge qa_n\qquad(n\ge1).
\end{equation}
We also assume that $\mu$ is a Borel probability measure on $\R^d$ with polynomial Fourier decay: for some $C_\mu,\eta>0$ and all $\xi\in\R^d$,
\begin{equation}\label{eq:fourier-decay}
 |\widehat\mu(\xi)|\le C_\mu(1+\|\xi\|)^{-\eta},
 \quad\text{where}\quad
 \widehat\mu(\xi)=\int e^{-2\pi i\xi\cdot x}\,d\mu(x).
\end{equation}
This includes normalized Lebesgue measure on any bounded cube. For a
singular example, let $\nu_\lambda$ be the Bernoulli convolution given by
the law of $(1-\lambda)\sum_{n\ge0}\varepsilon_n\lambda^n$, where the
$\varepsilon_n$ are independent and uniform on $\{0,1\}$. For every
$\lambda\in(0,1/2)$ its support is a Cantor set of Lebesgue measure zero.
The Erd\H{o}s--Kahane theorem gives polynomial Fourier decay for
Lebesgue-almost every such $\lambda$; see
Shmerkin~\cite[Proposition~2.3]{Shmerkin2014}.
For such $\lambda$, $\nu_\lambda^{\otimes d}$ is singular and satisfies
\eqref{eq:fourier-decay} in every dimension.

A convex body is a compact, convex subset of $\R^d$ with nonempty interior.
For a fixed convex body $\mathcal C\subset\R^d$, choose $r_0>0$
with $4r_0\operatorname{diam}\mathcal C<1$ and put
$$
 V_N^{\mathcal C}(x)=\sup\{r^d\vol(\mathcal C):0<r\le r_0,\
 (t+r\operatorname{int}\mathcal C)\cap A_N(x)=\varnothing
 \text{ for some }t\in\T^d\}.
$$
\begin{theorem}\label{thm:fixed-shape}
Assume \eqref{eq:lacunarity} and \eqref{eq:fourier-decay}.
For each fixed convex body $\mathcal C\subset\R^d$,
$$
 \lim_{N\to\infty}\frac{NV_N^{\mathcal C}(x)}{\log N}=1
 \qquad\text{for }\mu\text{-almost every }x\in\R^d.
$$
\end{theorem}

For the supremum over all convex bodies, define
$$
 G_N^*(x)=\sup\{\vol(K):K\subset[0,1]^d\text{ is a convex body},\ \operatorname{int}K\cap A_N(x)=\varnothing\}.
$$
\begin{theorem}\label{thm:all-convex}
Assume \eqref{eq:lacunarity} and \eqref{eq:fourier-decay}.  Then
$$
 \lim_{N\to\infty}\frac{NG_N^*(x)}{\log N}=d
 \qquad\text{for }\mu\text{-almost every }x\in\R^d.
$$
\end{theorem}
In particular, both theorems hold for Lebesgue-almost every $x\in\R^d$. We next consider the analogous extremal problem for independent uniform points in a fixed convex body.
\begin{theorem}\label{thm:iid}
Let $\Omega\subset\R^d$ be a convex body, let $Y_1,Y_2,\ldots$ be
independent uniform points in $\Omega$, and put
\[
 V_N^*=\sup\{\vol(K):K\subset\Omega\text{ is a convex body},\
          \operatorname{int}K\cap\{Y_1,\ldots,Y_N\}=\varnothing\}.
\]
Then
\[
 \lim_{N\to\infty}\frac{N V_N^*}{\vol(\Omega)\log N}=d
 \qquad\text{almost surely}.
\]
\end{theorem}

\subsection{Previous work and consequences}
In dimension one, let $G_N(x)$ denote the maximal gap of
$A_N(x)$. For lacunary integer sequences,
Chow--Technau~\cite[Theorem~3.1]{ChowTechnau2024} proved
$G_N\ll(\log N)^{3+\varepsilon}/N$ for $\mu$-a.e.\ $x$, where
$\mu$ is a probability measure on $[0,1]$ satisfying
\eqref{eq:fourier-decay}.\footnote{For nonnegative quantities, $A\ll B$ and $B\gg A$ mean
$A\le CB$ for a constant independent of the asymptotic parameter;
$A\asymp B$ means both $A\ll B$ and $B\ll A$. Subscripts indicate
permitted dependence of the constant on fixed parameters.}
Stefanescu~\cite[Theorems~2.2 and~2.4]{Stefanescu2025Adv}
improved the exponent to $2+\varepsilon$ for positive real lacunary
sequences under the same measure hypothesis, and proved that some
dilation parameter satisfies $G_N\ll(\log N)/N$ for all sufficiently
large $N$. In higher dimensions,
Stefanescu~\cite[Theorem~1]{Stefanescu2025MZ} proved
$G_N^*\ll(\log N)^2/N$ for positive real lacunary sequences,
uniformly over all convex bodies, for $\mu$-a.e.\ $x\in[0,1]^d$,
where $\mu$ is a probability measure satisfying \eqref{eq:fourier-decay}.

Peres--Yang~\cite[Theorems~1.1 and~1.2]{PeresYang2026} proved that,
for lacunary integer sequences satisfying \eqref{eq:lacunarity} and for
Lebesgue-a.e.\ $x$,
$$
 \frac12\le\liminf_{N\to\infty}\frac{NG_N(x)}{\log N}
 \le\limsup_{N\to\infty}\frac{NG_N(x)}{\log N}
 \le\frac{q+1}{q-1}.
$$
They obtained the limit $1$ when $a_n\mid a_{n+1}$ for every $n$.
Taking $\mathcal C$ to be an interval in \cref{thm:fixed-shape}
gives the limit $1$ for every positive real sequence satisfying
\eqref{eq:lacunarity}, for $\mu$-a.e.\ $x$ under
\eqref{eq:fourier-decay}. In higher dimensions, \cref{thm:all-convex}
sharpens Stefanescu's bound to
$G_N^*(x)\sim d(\log N)/N$, while \cref{thm:fixed-shape} gives the
constant $1$ for homothets of any fixed convex body.

Theorem~\ref{thm:fixed-shape} also determines the covering radius of
$A_N(x)$ with respect to any fixed norm. Let $\mathcal K$ be its unit
ball and put $\omega=\vol(\mathcal K)$. The covering radius is
\[
 \rho_N(x)=\inf\left\{r>0:
 \T^d=\bigcup_{n=1}^N\bigl((a_nx+r\mathcal K)\bmod1\bigr)\right\}.
\]
Since $V_N^{\mathcal K}(x)=\omega\rho_N(x)^d$ eventually for
$\mu$-almost every $x$,
$$
 \rho_N(x)\sim\left(\frac{\log N}{\omega N}\right)^{1/d}
 \quad\text{for }\mu\text{-almost every }x\in\R^d.
$$
Uniform covering statements of this type are related to the
eventually-always-hitting and uniform approximation problems studied by
Kirsebom--Kunde--Persson~\cite{KirsebomKundePersson2020},
Koivusalo--Liao--Persson~\cite{KoivusaloLiaoPersson2023}, and
He--Liao~\cite{HeLiao2024}.
A related circle covering problem for shrinking intervals centered at
lacunary dilates was studied by
Hauke--Shubin--Stefanescu--Zafeiropoulos~\cite{HSSZ2026}.

Further results on selecting dilation parameters include
Alon--Peres~\cite{AlonPeres1992},
Konyagin--Ruzsa--Schlag~\cite{KRS2000},
Peres--Schlag~\cite{PeresSchlag2010}, and
Moshchevitin~\cite{Moshchevitin2012}. For local spacing statistics,
Rudnick--Zaharescu~\cite{RudnickZaharescu2002} and
Chaubey--Yesha~\cite{ChaubeyYesha2022} established Poissonian phenomena
for lacunary dilations.

For independent uniform points, Devroye~\cite{Devroye1982} obtained
the one-dimensional maximal-gap asymptotics, while
Janson~\cite{Janson1987} proved the higher-dimensional fixed-shape
analogue, again with leading constant $1$.
Deheuvels--Einmahl--Mason--Ruymgaart~\cite[Theorem~1(II) and
Example~2]{DEMR1988} showed that, for iid uniform points in the unit
cube and fixed dimension, the same constant $1$ persists for all
axis-parallel boxes, even with arbitrary aspect ratios.

In dimension two, \cref{thm:iid} gives the exact constant $2$ for
the maximal volume of an empty convex body.
Arizmendi--Salazar~\cite[Theorem~1.2 and Section~6]{ArizmendiSalazar2016}
proved high-probability bounds with leading constants $1$ and $4$
and conjectured the value $1$. Thus \cref{thm:iid} disproves that
conjecture on the unit square.
Aldous~\cite[H15.1]{Aldous1989} had already asked about the maximal
volume of an empty convex body for a Poisson point process in the unit square.
For planar rectangles of arbitrary orientation,
Arizmendi--Salazar~\cite[Theorem~4.1]{ArizmendiSalazar2016} proved
the upper coefficient $2$.
Auer--Hornik~\cite{AuerHornik1994} obtained earlier covering
estimates that imply the same upper coefficient; see also
\cite[pp.~6--7]{AuerHornik1996}.

\subsection{Lacunarity and sparsity}
The following theorem shows that lacunarity cannot be replaced by sparsity
on sublacunary scales, even under the strong condition that each interval
$(b_k,b_{k+1}]$ contains at most one term of the sequence.

\begin{theorem}\label{thm:sparse}
Let $(b_k)_{k\ge1}$ be positive, strictly increasing and unbounded, with
$b_{k+1}/b_k\to1$.  There is a strictly increasing sequence $(a_n)$ of
positive integers such that
$$
 \#\{n:b_k<a_n\le b_{k+1}\}\le1
 \qquad(k\ge1),
$$
and, for Lebesgue-almost every $x\in\T$,
$$
 \liminf_{N\to\infty}NG_N(x)=1,
 \qquad
 \limsup_{N\to\infty}\frac{NG_N(x)}{\log N}=\infty.
$$
\end{theorem}
For example, one may take $b_k=2^{k^\alpha}$ with $0<\alpha<1$.
The construction uses widely separated finite arithmetic progressions.
For Lebesgue-almost every dilation parameter, every sufficiently late
arithmetic block has an initial segment with nearly uniform spacing and
a longer initial segment with an unusually large maximal gap.

\subsection{Proof strategy and organization}
The leading constants reflect geometric approximation counts. At volume
scale $s$, we use $O(s^{-1})$ inner approximants for homothets of a fixed
body and $O(s^{-d}(\log(1/s))^{C_d})$ for arbitrary convex bodies.
For lacunary dilates, the difficulty is to obtain the avoidance bounds
needed to exploit these counts without independence.

We adapt the block method of Peres--Yang~\cite{PeresYang2026} and the
scale-separation method of Peres--Schlag~\cite[Section~3]{PeresSchlag2010}.
The indices are divided into active blocks separated by buffers of
omitted indices. For each block, we use a local average $h_b$ of its
hit indicator. Scale separation allows the indicator to be replaced by
$h_b$ in integrals weighted by the preceding survival history. This
transfer follows from translation averaging for Lebesgue measure and
from trigonometric approximation and Fourier decay for $\mu$.
Compensated survivor weights then give exponential avoidance bounds.
For generic targets, the local hitting rate is asymptotically the
independent-point rate.

Targets with large pair intersections require an arithmetic argument.
Such intersections force bounded-coefficient relations among frequencies
and restrict the target centers. We handle large residuals in these
relations by frequency separation and small ones by shifted frequency
sequences.
In higher dimensions, rational-rank stratification controls the number
of exceptional targets, while uniform positive hitting estimates handle
the remaining cases. Together these arguments complete the upper bounds.

For the lower bounds, we average damped survivor weights over
translations of a fixed body or over the intercepts and slopes of thin
slabs. Near pairs of target parameters are controlled by incidence
estimates. For thin slabs, far pairs are controlled by transverse
intersection estimates.
The weighted survivor mass has expectation $(1+o(1))e^{-Ns}$ and
vanishing relative variance below the respective thresholds. The
expected weighted mass of targets that survive the active blocks but
are hit at omitted indices is $o(e^{-Ns})$. Hence, with probability
tending to one, at least one target avoids all $N$ points.

Section~\ref{sec:one-dimensional} proves the one-dimensional Lebesgue
case. Section~\ref{sec:iid-points} proves \cref{thm:iid} and establishes
the geometric estimates used later. Sections~\ref{sec:prelim}
and~\ref{sec:all-convex-geometry} develop the Fourier transfer and local
hitting estimates. Sections~\ref{sec:all-convex-upper}
and~\ref{sec:all-convex-lower} prove
\cref{thm:fixed-shape,thm:all-convex} in full generality.
Section~\ref{sec:general-sparse-counterexample} proves \cref{thm:sparse}.

\section{The one-dimensional Lebesgue argument}\label{sec:one-dimensional}
We prove \cref{thm:fixed-shape} in dimension one for Lebesgue measure.
Since convex bodies in $\R$ are intervals, it suffices to show
\[
 \frac{NG_N(x)}{\log N}\longrightarrow1
 \quad\text{for Lebesgue-a.e. }x\in\R,
\]
where $G_N(x)$ denotes the maximal gap of the points
$a_1x,\ldots,a_Nx\bmod1$ on $\T$.

We adapt the block method and local counting estimates of
Peres--Yang~\cite{PeresYang2026}. Local averaging controls compensated
survival weights in both bounds. To obtain the sharp constant for general
positive real lacunary sequences, we localize large pair intersections.

Fix a bounded interval $Q_0\subset\R$ of positive length. For a measurable
set $E\subset\R$ and an integrable function $f$, put
\[
 \mathbb P(E)=\frac{|E\cap Q_0|}{|Q_0|},\qquad
 \mathbb E f=\frac1{|Q_0|}\int_{Q_0}f(x)\,dx.
\]
For $t\in\T$ and $0<s<1$, let $I_t(s)=(t,t+s)\pmod1$ be the
open arc of length $s$ starting at $t$. Write
$\|y\|_{\R/\Z}=\min_{j\in\Z}|y-j|$ for distance to the nearest integer.
Constants in the estimates may depend on $q,Q_0$ and parameters fixed
below. The threshold for sufficiently large $N$ may also depend on $a_1$.
Unless stated otherwise, the estimates are uniform in the block indices
and arc positions; pointwise estimates are uniform in $x$.
We will prove that, for each $0<c<1<c'<2$, both
\[
 \mathbb P\{G_N<c\log N/N\},\qquad
 \mathbb P\{G_N>c'\log N/N\}
\]
are bounded by a negative power of $N$. The almost-sure conclusion will
then follow by Borel--Cantelli and monotonicity.

\subsection{Block averaging and transfer}
Let $\Phi=\tfrac12\1_{[-1,1]}$. For $T\ge1$, put
$\Phi_T(u)=T\Phi(Tu)$, so that
\[
 (\Phi_T*f)(x)=\frac T2\int_{x-1/T}^{x+1/T}f(y)\,dy.
\]
Fix an integer $K\ge30$, put $q_0=(q+1)/2$, and choose
the buffer length
\[
 g=\left\lceil\frac{4K\log N}{\log q_0}\right\rceil.
\]
Let $L=\lfloor N^\xi\rfloor$ for some fixed $0<\xi<1$, to be
specified in the upper- and lower-bound proofs. The active blocks are
\[
 \mathcal I_b=\{(b-1)(L+g)+g+1,\ldots,b(L+g)\},
 \qquad 1\le b\le B,
\]
where $B=\lfloor N/(L+g)\rfloor$. There are $g$ omitted indices before
each block, and the final incomplete block is also omitted. Thus
$N-BL=O(Ng/L+L+g)=o(N/\log N)$. Set $a_0=1$ and define
\[
 T_b=N^K\max\{1,a_{(b-1)(L+g)}\}.
\]
Since $q_0^g\ge N^{4K}$, for all sufficiently large $N$,
\[
 \min_{n\in\mathcal I_b}a_n\ge N^KT_b.
\]
The same estimate holds uniformly for $q_0$-lacunary sequences
whose first term is at least $1$, with $T_b$ defined accordingly.
For an arc $I$, define
\[
 H_b(x,I)=\1_{\{a_nx\bmod1\in I\text{ for some }n\in\mathcal I_b\}},
 \qquad h_b(x,I)=(\Phi_{T_b}*H_b(\cdot,I))(x).
\]
When $s$ is fixed, abbreviate $H_b(x,t)=H_b(x,I_t(s))$ and
$h_b(x,t)=h_b(x,I_t(s))$.

For an arbitrary arc $I$ of length $v$, counting full periods and the
two possible incomplete periods gives, uniformly in $x$,
\[
 \frac T2\int_{x-1/T}^{x+1/T}\1_I(ay\bmod1)\,dy
 =v+O(T/a).
\]
Since $a_n\ge N^KT_b$ for $n\in\mathcal I_b$, it follows that
\[
 0\le h_b(x,I)\le\sum_{n\in\mathcal I_b}\left(v+\frac{T_b}{a_n}\right)
 \le L(v+N^{-10}).
\]
For applications with $v\le s<2\log N/N$, put
$p=L(s+N^{-10})=o(1)$; then $h_b(x,I)\le p$. Since $BL\le N$, we have
$Bp\le N(s+N^{-10})\le3\log N$ and
$Bp^2\le NL(s+N^{-10})^2=o(1)$.

For an arc $I$ of length at most $s$, define the compensated survival
weight
\[
 W_0(x,I)=1,\qquad
 W_b(x,I)=\prod_{j=1}^b(1-H_j(x,I))(1+h_j(x,I)).
\]
Let $I_1,I_2$ be arcs of lengths at most $s$, and let
$J_1,J_2\subset\{1,\ldots,B\}$. Consider weights of the form
\[
 F(x)=\prod_{\ell=1}^2\prod_{j\in J_\ell}
       (1-H_j(x,I_\ell))(1+h_j(x,I_\ell)).
\]
Empty products and constant multiples in $[0,1]$ are allowed.
When applying the averaging estimate to
block $b$, we require $J_1,J_2\subset\{1,\ldots,b-1\}$. Each of the
two products is at most $(1+p)^B\le e^{Bp}\le N^3$, so
$0\le F\le N^6$. We also allow one additional factor
$\1_E(z_0x\bmod1)$, where $z_0\in\R\setminus\{0\}$ and $E$ is an arc.
For such a weight $F$, set
\[
 a_* = \max\bigl(\{1\}\cup
       \{a_n:n\in\mathcal I_j,\ j\in J_1\cup J_2\}\bigr).
\]
If $F$ contains the additional indicator, include $|z_0|$ in this
maximum. Thus $a_*=1$ if no block or indicator factors occur. In the
single-insertion estimate below, $z$ is the frequency of a new
indicator and is not included in $a_*$.
\begin{lemma}\label{lem:od-transfer}
For a weight $F$ as above, with $J_1,J_2\subset\{1,\ldots,b-1\}$,
and any arc $I$ of length at most $s$, one has
\begin{equation}\label{od:transfer}
 \mathbb E[FH_b(\cdot,I)]=\mathbb E[Fh_b(\cdot,I)]+O(N^{-10})
 \quad\text{whenever }T_b\ge N^Ka_*.
\end{equation}
For a weight $F$ without an inserted indicator,
\begin{equation}\label{od:single-insertion}
 \mathbb E[F\1_E(zx\bmod1)]
 =|E|\mathbb EF+O(N^{-10})
 \quad\text{if }|z|\ge N^Ka_*.
\end{equation}
The estimates are uniform in the positions of the arcs and remain
valid if some preceding block factors are omitted.
\end{lemma}

\begin{proof}
Let $H:\R\to[0,1]$ be measurable, and suppose $G\in L^1(\R)$ satisfies
$\int_\R|G(x+u)-G(x)|\,dx\le M|u|$ for all $|u|\le1$,
where $M\ge0$ is a constant. Symmetry of $\Phi_T$ and Fubini give
\begin{equation}\label{od:averaging}
 \int G(H-\Phi_T*H)\,dx
 =\int H(G-\Phi_T*G)\,dx,
 \qquad
 \left|\int G(H-\Phi_T*H)\,dx\right|\le\frac{M}{2T}.
\end{equation}
Write $\operatorname{TV}_J(f)$ for the total variation of $f$ on an interval $J$.
On $Q_0+[-2,2]$, the indicator $\1_I(ax\bmod1)$ has at most
$C(1+a)$ jumps. The discontinuities of $H_j$ occur among those of the
individual indicators, and convolution does not increase total variation.
Consequently, on $Q_0+[-1,1]$,
\[
 \operatorname{TV}_{Q_0+[-1,1]}(H_j(\cdot,I_\ell))
 +\operatorname{TV}_{Q_0+[-1,1]}(h_j(\cdot,I_\ell))
 \ll\sum_{n\in\mathcal I_j}(1+a_n).
\]
The product variation inequality and $0\le F\le N^6$ now give
\[
 \begin{aligned}
 \operatorname{TV}_{Q_0+[-1,1]}(F)
 \ll N^6\left(1+|z_0|
   +\sum_{\ell=1}^2\sum_{j\in J_\ell}
      \sum_{n\in\mathcal I_j}(1+a_n)\right)\ll N^7a_*.
 \end{aligned}
\]
Here the term $|z_0|$ is omitted when no additional indicator occurs;
there are at most $2N$ summands in the triple sum, each bounded by
$1+a_*$. The restriction to $Q_0$ adds at most $2N^6$ to the total
variation. The translation inequality for functions of bounded
variation therefore yields
\[
 \begin{aligned}
 \int_\R\left|\1_{Q_0}(x+u)F(x+u)
                   -\1_{Q_0}(x)F(x)\right|\,dx\le |u|\operatorname{TV}_\R(\1_{Q_0}F)
 \ll N^7a_*|u|\qquad(|u|\le1).
 \end{aligned}
\]
Apply \eqref{od:averaging} with
$G=|Q_0|^{-1}\1_{Q_0}F$ and $H=H_b(\cdot,I)$ to obtain
\eqref{od:transfer}. We use this estimate in place of independence
between blocks.

For the second estimate, the periodic function
$\1_E(zx\bmod1)-|E|$ has a primitive bounded by $C/|z|$.
Integration by parts and the preceding variation bound prove
\eqref{od:single-insertion}.
\end{proof}

\subsection{Pair localization and hitting estimates}
For an interval $Q\subset\R$ of length $\rho$, write
$\mathbb P_Q(E)=|E\cap Q|/\rho$, and let $\mathbb E_Q$ denote expectation
with respect to $\mathbb P_Q$. Throughout this subsection, $I$ is an
arc of length $v$, and $a<b$ are positive. Counting the periods of
$ax\bmod1$ that meet $Q$ gives
\[
 \mathbb P_Q(ax\bmod1\in I)=v\bigl(1+O((a\rho)^{-1})\bigr).
\]
\begin{lemma}\label{lem:od-pairs}
If $a\rho\ge4$, then
\begin{equation}\label{od:pair-coarse}
 \mathbb P_Q(ax\bmod1\in I,bx\bmod1\in I)
 \le C(v^2+va/b).
\end{equation}
Moreover, for every $\epsilon>0$ there are constants $H,C_0,T_0$,
depending only on $q,\epsilon$, such that, for all sufficiently small
$v$, if $b/a\ge q$, $a\rho\ge T_0$, and the pair intersection probability
is greater than $\epsilon v$, then some integers $m,n$ satisfy
\begin{equation}\label{od:pair-return}
 1\le m<n\le H,\qquad |na-mb|\rho\le C_0v.
\end{equation}
Let $t_I\in\T$ be the center of $I$. For $1\le D\le H$, put
\begin{equation}\label{od:incidence-arcs}
 E_D(I)=\{y\in\T:\|y-Dt_I\|_{\R/\Z}\le C_1v\},
\end{equation}
where $C_1=C_1(q,\epsilon)$ is sufficiently large. If the relation
\eqref{od:pair-return} holds and both dilates hit $I$ at $x_0\in Q$,
then
\begin{equation}\label{od:pair-location}
 z x\bmod1\in E_D(I)\quad\text{for every }x\in Q,
 \qquad z=na-mb,\quad D=n-m.
\end{equation}
\end{lemma}

\begin{proof}
There are at most $a\rho+2$ slow periods. In each of them the
slow hit is an interval of length $v/a$, which meets at most
$(b/a)v+2$ fast hit intervals, each of length $v/b$. Hence
\[
 \mathbb P_Q(ax\in I,\,bx\in I)
 \le\frac{a\rho+2}{\rho}\left(\frac{bv}{a}+2\right)\frac vb
 \ll v^2+\frac{av}{b},
\]
where $a\rho\ge4$ is used in the last inequality. This proves
\eqref{od:pair-coarse}.

For the second assertion, choose a unit interval containing a real
lift of $I$, and use its integer translates as period cells. By
\eqref{od:pair-coarse}, we may suppose $b/a\le R$ for a constant
$R=R(\epsilon)$. Put $r=b/a$ and $T=a\rho$.
Discarding the at most two partial slow periods costs at most $2v/T$.
On a full slow period write $ax=k+y$, with $y$ in the chosen lift of
$I$. If the pair probability exceeds $\epsilon v$, then, for some
$y$ in that lift, at least the fraction $\epsilon'=\epsilon/4$ of
the $T+O(1)$ consecutive integers $k$ satisfy
$rk+ry\bmod1\in I$.

We use the following estimate, valid for every $r\in\R$.
Let $k$ range over $P$ consecutive integers. If at least $\epsilon'P$
of the points $rk\bmod1$ lie in an arc of length $v$, then, for
sufficiently large $P$ and small $v$,
\begin{equation}\label{od:rotation-count}
 P\|mr\|_{\R/\Z}\le Cv\quad\text{for some }1\le m\le2/\epsilon'.
\end{equation}
Some two successive hitting indices have difference
$m\le2/\epsilon'$. Thus $\delta=\|mr\|_{\R/\Z}\le v$.
If $\delta=0$ there is nothing to prove. Otherwise split the indices
into their $m$ residue classes. In each class, the points advance by
$\delta$ or $-\delta$; a progression of $L$ such points meets the arc
at most $(L\delta+2)(v/\delta+1)$ times. Summing gives
\[
 \epsilon'P\le2Pv+2m(v/\delta+1).
\]
For $v\le\epsilon'/8$ and $P\ge16/(\epsilon')^2$, multiplying
by $\delta$ gives
\[
 (\epsilon'-2v)P\delta\le2m(v+\delta)\le4mv,
\]
so $P\delta\le C_{\epsilon'}v$.
Apply the observation and let $n$ be the nearest integer to $mr$.
Then $m,n$ are bounded and $T|mr-n|\le Cv$. Since
$m(r-1)\ge q-1$ and $|mr-n|\ll_{\epsilon'}v/P$, we have $n>m$
for sufficiently small $v/P$. This proves \eqref{od:pair-return}.

Finally, $zx_0\bmod1$ differs from $Dt_I$ by at most
$(m+n)v/2$, and $|z(x-x_0)|\le C_0v$ for all $x\in Q$.
Choosing $C_1$ sufficiently large proves \eqref{od:pair-location}.
\end{proof}

We apply the preceding estimates to blocks of $L$ consecutive
indices, writing $a_{\min}$ for the smallest frequency in a block.
For each $\tau>0$, choose $M$ sufficiently large in terms of $q,\tau$
and then $\epsilon>0$ sufficiently small in terms of $M,\tau$.
Partition each block into consecutive groups of $M$ indices,
discarding fewer than $M$ indices at the end. A group is called
\emph{bad} (relative to $Q$ and $I$) if some pair in it has
intersection probability greater than $\epsilon v$; write $B_Q(I)$
for the number of bad groups.
Take $H$ from \cref{lem:od-pairs} for this choice of $\epsilon$.
For each group, list all choices of $i<j$ in that group and
$1\le m<n\le H$. A choice $r=(i,j,m,n)$ has two associated numbers
\[
 z_r=na_i-ma_j,\qquad D_r=n-m.
\]
Retain all choices, including those with equal residuals $z_r$.
The list has size at most $CN$ over all active blocks. Write
$\mathscr R_b$ for the choices in block $b$, and define
\begin{equation}\label{od:beta-defined}
 \beta_b(x,I)=\sum_{r\in\mathscr R_b}
       \1_{E_{D_r}(I)}(z_rx\bmod1),\qquad
 \beta(x,I)=\sum_b\beta_b(x,I).
\end{equation}

\begin{lemma}\label{lem:od-block-hitting}
If $Lv\le1$ and $a_{\min}\rho$ is sufficiently large, there is a
constant $\kappa=\kappa(q)>0$ such that
\begin{equation}\label{od:positive-hit}
 \mathbb P_Q(\text{at least one hit of }I)\ge\kappa Lv.
\end{equation}
If also $Lv\to0$ and $L\to\infty$, then
\begin{equation}\label{od:generic-hit}
 \mathbb P_Q(\text{at least one hit of }I)
 \ge(1-\tau)Lv-Mv\,B_Q(I).
\end{equation}
In particular, for the active blocks and local averages defined above,
\begin{equation}\label{od:hit-rates}
 \kappa Lv\le h_b(x,I),\qquad
 h_b(x,I)\ge(1-\tau)Lv-Mv\beta_b(x,I).
\end{equation}
\end{lemma}

\begin{proof}
Let $Y$ count the hits in a block. The marginal estimate gives
$\mathbb E_QY\ge Lv/2$, and \eqref{od:pair-coarse} gives
\[
 \mathbb E_QY^2\le C_q(Lv+L^2v^2),
 \qquad \sum_{i<j}a_i/a_j\le L\sum_{k\ge1}q^{-k}.
\]
Now use $(\mathbb E_QY)^2\le\mathbb P_Q(Y>0)\mathbb E_QY^2$
to obtain \eqref{od:positive-hit}.

For a group that is not bad, the first two terms of
inclusion--exclusion give
\[
 \begin{aligned}
 \mathbb P_Q(\text{group hits }I)
 \ge Mv\bigl(1-O((a_{\min}\rho)^{-1})\bigr)
       -\binom M2\epsilon v\ge(1-\tau/4)Mv,
 \end{aligned}
\]
by first taking $a_{\min}\rho$ sufficiently large and then
$\epsilon$ sufficiently small. For pairs in different groups, sum the
frequency ratios across the group boundaries:
\[
 \sum_{\substack{i<j\\i,j\text{ in different groups}}}\frac{a_i}{a_j}
 \le\frac LM\sum_{u\ge0}\sum_{w\ge1}q^{-(u+w)}
 \ll_q\frac LM.
\]
Thus the cross-group pair intersections contribute at most
$C L^2v^2+C_qLv/M$. By inclusion--exclusion over the good groups,
the bad groups and the discarded final group cost at most
$Mv B_Q(I)+Mv$. Since $Lv\to0$ and $L\to\infty$, choose $M$
sufficiently large, followed by $\epsilon$ sufficiently small, to
absorb the remaining errors into $\tau Lv$. This proves
\eqref{od:generic-hit}.

By \eqref{od:pair-location}, every bad group contributes a summand
equal to one throughout $Q$. Applying the preceding estimates on the
averaging interval gives \eqref{od:hit-rates}. Since
$a_{\min}/T_b\ge N^K$, every averaging interval satisfies the
required lower bound on $a_{\min}\rho$.
\end{proof}

\subsection{Upper bound}
For a set of retained blocks $\mathcal B\subseteq\{1,\ldots,B\}$,
put
\[
 W_b^{\mathcal B}(x,I)
 =\prod_{\substack{j\in\mathcal B\\j\le b}}
             (1-H_j(x,I))(1+h_j(x,I)).
\]
Thus $W_b^{\{1,\ldots,B\}}=W_b$ as defined above.

\begin{lemma}\label{lem:od-uniform-avoidance}
Let $I$ be an arc of length $v\le s<2\log N/N$.
Use the separated-block construction above, with
$L=\lfloor N^\xi\rfloor$ for fixed $0<\xi<1$.
Let $\mathcal B\subseteq\{1,\ldots,B\}$ be obtained by omitting whole
active blocks, and suppose $|\mathcal B|L=N-o(N)$. Then, for a constant
$\kappa>0$ depending only on $q$,
\begin{equation}\label{od:uniform-avoidance}
 \mathbb P\{H_b(\cdot,I)=0\text{ for all }b\in\mathcal B\}
 \le\exp(-(\kappa-o(1))Nv)+O(N^{-9}).
\end{equation}
The bound is uniform in the position of $I$. It also holds uniformly
for $q_0$-lacunary sequences whose first frequency is at least one,
using the same block construction and the corresponding averaging
scales, where $q_0=(q+1)/2$.
\end{lemma}

\begin{proof}
For $b\notin\mathcal B$, $W_b^{\mathcal B}=W_{b-1}^{\mathcal B}$. For $b\in\mathcal B$, since
$(1-H_b)(1+h_b)=1+(h_b-H_b)-H_bh_b$,
the averaging estimate \eqref{od:transfer} gives
\[
 \begin{aligned}
 \mathbb EW_b^{\mathcal B}-\mathbb EW_{b-1}^{\mathcal B}
 =\mathbb E[W_{b-1}^{\mathcal B}(h_b-H_b)]
   -\mathbb E[W_{b-1}^{\mathcal B}H_bh_b]\le O(N^{-10}).
 \end{aligned}
\]
Hence $\mathbb EW_B^{\mathcal B}\le1+O(BN^{-10})=1+O(N^{-9})$.
Moreover, $h_b\le L(v+N^{-10})$ implies
\[
 \sum_{b\in\mathcal B} h_b^2
 \le |\mathcal B|L^2(v+N^{-10})^2
 \le NL(v+N^{-10})^2=o(1).
\]
If $I$ is avoided by all retained blocks, then
\[
 W_B^{\mathcal B}(x,I)\ge\exp\left(\sum_{b\in\mathcal B}h_b(x,I)-o(1)\right).
\]
The lower bound in \eqref{od:hit-rates} gives
$\sum_{b\in\mathcal B}h_b(x,I)\ge\kappa|\mathcal B|Lv
=(\kappa-o(1))Nv$.
This proves \eqref{od:uniform-avoidance}. The same argument applies
to the stated $q_0$-lacunary sequences after decreasing $\kappa$
if necessary.
\end{proof}

Fix $1<c<2$, put $s=c\log N/N$, and choose $0<\varepsilon<1$ such
that $(1-\varepsilon)c>1$. Put $v=(1-\varepsilon)s$.
Choose $\tau>0$ with $(1-\tau)(1-\varepsilon)c>1$, and take
$M,H$ as in the preceding subsection. Then choose $\delta>0$ so that
$(1-\tau-M\delta)(1-\varepsilon)c>1$.
Choose equally spaced centers on $\T$, with spacing at most
$\varepsilon s$, and let $\mathscr I$ be the arcs of length $v$
at these centers. There are $O(s^{-1})$ such arcs, their center
spacing is bounded below by a fixed multiple of $s$, and every
arc of length $s$ contains one of them.
Use $L=\lfloor\sqrt N\rfloor$ and the separated active blocks
defined above, retaining all full active blocks. Thus the relevant
weight is $W_b(x,I)=W_b^{\{1,\ldots,B\}}(x,I)$.
By the proof of \cref{lem:od-uniform-avoidance}, on avoidance of $I$,
$W_B(x,I)\ge\exp(\sum_bh_b(x,I)-o(1))$ and
$\mathbb EW_B\le1+O(N^{-9})$.

If $\beta(x,I)\le\delta N$, \eqref{od:hit-rates} gives
\[
 \begin{aligned}
 \sum_{b=1}^Bh_b(x,I)
 \ge(1-\tau)BLv-Mv\beta(x,I)\ge(1-\tau-M\delta-o(1))Nv.
 \end{aligned}
\]
Consequently, on avoidance of $I$,
$W_B(x,I)\ge\exp((1-\tau-M\delta-o(1))Nv)$.
Markov's inequality and summation over the $O(s^{-1})$ arcs yield
\begin{equation}\label{od:ordinary-upper}
 \mathbb P\{\exists I\in\mathscr I:\ I\text{ avoided},\
                              \beta(x,I)\le\delta N\}
 \le N^{1-(1-\tau-M\delta)(1-\varepsilon)c+o(1)}.
\end{equation}

It remains to treat the arcs with $\beta(x,I)>\delta N$.
Let $\mathcal E_I=\{x\in Q_0:a_nx\bmod1\notin I\text{ for }1\le n\le N\}$, and let
$\mathscr R$ be the full list of choices defined above. Pointwise,
\begin{equation}\label{od:incidence-sum}
 \1_{\{\exists I\in\mathscr I:\mathcal E_I,\ \beta(x,I)>\delta N\}}
 \le \frac1{\delta N}\sum_{r\in\mathscr R}\sum_{I\in\mathscr I}
       \1_{\mathcal E_I}\1_{E_{D_r}(I)}(z_rx\bmod1).
\end{equation}
We prove that each inner sum has expectation $O(N^{-\gamma})$
for some $\gamma>0$, uniformly in $r$. Since $|\mathscr R|=O(N)$,
this will finish the argument. Write $z=z_r$ and $D=D_r$.

First suppose $|z|>N^{2K}$. Remove every active block whose
frequency range intersects $[|z|N^{-6K},|z|N^{6K}]$, and let
$\mathcal B_z\subset\{1,\ldots,B\}$ be the retained block indices.
This deletes $O(\log N)+2L=o(N)$ indices: only the first and last
intersecting blocks can have frequencies outside that interval.
All retained blocks lie below or above it.
For the first retained block above the interval, if one exists, replace
its averaging scale by
\[
 \widetilde T_b=N^K\max\{1,|z|,a_-\},
\]
where $a_-$ is the largest retained frequency below the interval,
or $1$ if there is none. Keep the original scales for the other
retained blocks. Let $\widetilde h_b$ be the resulting local averages,
and define
\[
 \widetilde W_z(x,I)
 =\prod_{b\in\mathcal B_z}(1-H_b(x,I))(1+\widetilde h_b(x,I)).
\]
Insert $J(x)=\1_{E_D(I)}(zx\bmod1)$ between the retained blocks
below and above the interval. By \eqref{od:single-insertion},
\[
 \mathbb E[FJ]=|E_D(I)|\mathbb E F+O(N^{-10}),
\]
where $F$ is the preceding product, or $1$ if there is none.
The frequency separation gives
$\min_{n\in\mathcal I_b}a_n\ge N^K\widetilde T_b$ for the first
retained block above the interval. Thus the proof of
\eqref{od:transfer} applies to the modified averages as well;
subsequent retained blocks already have the required separation.
Iterating before and after the insertion yields
\[
 \mathbb E[J\widetilde W_z]\le Cv+O(N^{-9}).
\]
Moreover $|\mathcal B_z|L=N-o(N)$ and each retained block satisfies
$\widetilde h_b(x,I)\ge\kappa Lv$ by \eqref{od:positive-hit}.
On $\mathcal E_I$, the sum of the squared averages is $o(1)$, so
\[
 \widetilde W_z(x,I)\ge\exp((\kappa-o(1))Nv).
\]
Consequently,
\[
 \mathbb P(\mathcal E_I\cap\{zx\bmod1\in E_D(I)\})
 \le Cv\exp(-(\kappa-o(1))Nv)+O(N^{-8}).
\]
Summing over $O(v^{-1})$ arcs gives the required polynomial decrease.

Now suppose $|z|\le N^{2K}$, including $z=0$.
The condition $zx\bmod1\in E_D(I)$ implies that the center of $I$
lies within $C_1v$ of one of the $D$ points
\[
 p_j(x)=(zx+j)/D\bmod1,\qquad 0\le j<D.
\]
For each such point only $O(1)$ grid arcs can satisfy this condition,
because their center spacing is comparable to $v$.
Fix an integer $J_0>4C_1+4$, and let
\[
 J_\ell=(\ell v/4-v/4,\ell v/4+v/4)\pmod1,
                 \qquad |\ell|\le J_0.
\]
Every arc $I$ whose center is within $C_1v$ of $p_j(x)$ contains
$p_j(x)+J_\ell$ for at least one of these finitely many $\ell$.
Thus
\[
 \sum_{I\in\mathscr I}\1_{\mathcal E_I}
                       \1_{E_D(I)}(zx\bmod1)
 \le C\sum_{j=0}^{D-1}\sum_{|\ell|\le J_0}
 \1_{\{((a_n-z/D)x-j/D)\bmod1\notin J_\ell\ (n\le N)\}}.
\]
Omit the initial indices until $a_n\ge N^{4K}+2|z|$; this costs
$O(\log N)$ indices. The remaining positive frequencies
$\widetilde a_n=a_n-z/D$ satisfy
$\widetilde a_{n+1}/\widetilde a_n\ge(q+1)/2$.
For $z\ge0$ the ratio is at least $q$; for $z<0$, set
$u=|z|/D$ and use
$(qa_n+u)/(a_n+u)\ge(q+1)/2$.
The fixed phase $-j/D$ is absorbed into the arc.
Apply \eqref{od:uniform-avoidance} to these frequencies and arcs of
length $v/2$. Each summand has probability
$\exp(-(\kappa-o(1))Nv/2)+O(N^{-9})$.
There are only finitely many summands, so the required uniform
$O(N^{-\gamma})$ bound follows also in this case.

The chosen $\tau,\delta$ make the exponent in
\eqref{od:ordinary-upper} negative. Combining it with
\eqref{od:incidence-sum} proves, for some $\alpha>0$,
\[
 \mathbb P\{G_N>c\log N/N\}=O(N^{-\alpha}).
\]

\subsection{Lower bound}
Fix $0<c<1$ and put
\[
 s=\frac{c\log N}{N},\qquad
 \xi=\frac{1-c}{2},\qquad L=\lfloor N^\xi\rfloor.
\]
Use the blocks, indicators $H_b(x,t)$, and averages $h_b(x,t)$ defined
above, with this choice of $L$. The one-point estimate gives
\[
 0\le h_b(x,t)\le p:=L(s+N^{-10}).
\]
Set $\theta=1-p$. For this choice of $L$, let $W_b$ denote the
full-block weight defined above, and put
\[
 U_b(x,t)=\theta^b W_b(x,I_t(s)),\qquad
 Z_b(x)=\int_\T U_b(x,t)\,dt.
\]
Equivalently, $U_0=1$ and
\[
 U_b=\theta U_{b-1}(1-H_b)(1+h_b).
\]
Since $\theta(1+h_b)\le(1-p)(1+p)\le1$, we have
$0\le U_b\le U_{b-1}\le1$. Thus $U_b(x,t)$ vanishes whenever
one of the first $b$ blocks hits $I_t(s)$, and $Z_B>0$ implies that
some arc of length $s$ avoids all active blocks.

We first estimate $\mathbb E Z_B$. The identity
\[
 (1-H_b)(1+h_b)=1+(h_b-H_b)-H_bh_b
\]
and the averaging estimate give
\[
 \bigl|\mathbb E[U_{b-1}(h_b-H_b)]\bigr|\le N^{-10},
 \qquad
 0\le\mathbb E[U_{b-1}H_bh_b]
 \le p\,\mathbb E[U_{b-1}H_b]
 \le p^2\mathbb E U_{b-1}+N^{-10}.
\]
These bounds are uniform in $t$ and imply
\[
 \theta(1-p^2)\,\mathbb EU_{b-1}-O(N^{-10})
 \le\mathbb EU_b
 \le\theta\,\mathbb EU_{b-1}+O(N^{-10}).
\]
Iterating gives
\begin{equation}\label{od:lower-first}
 \mathbb E U_b(\cdot,t)
 =\theta^b\bigl(1+O(Bp^2)\bigr)+O(N^{-8})
 \qquad(0\le b\le B).
\end{equation}
The number of indices outside the full blocks is
\[
 m_N=N-BL=O(Ng/L+L+g).
\]
With $g=O(\log N)$ and $L=N^\xi+O(1)$,
\[
 m_Ns\ll\frac{(\log N)^2}{L}+\frac{L\log N}{N}=o(1),
 \qquad
 Bp^2\ll\frac{L(\log N)^2}{N}=o(1).
\]
Since $Bp=Ns-m_Ns+O(N^{-9})$, we obtain
\begin{equation}\label{od:lower-mean}
 \mathbb E Z_B=\theta^B(1+o(1)),\qquad
 \theta^B=\exp\{-Ns+O(m_Ns+Bp^2)\}=N^{-c}(1+o(1)).
\end{equation}

We next estimate $\mathbb E Z_B^2$. For $t,u\in\T$, abbreviate $H=H_b(x,t)$, $H'=H_b(x,u)$,
$h=h_b(x,t)$, $h'=h_b(x,u)$ and
$F=U_{b-1}(x,t)U_{b-1}(x,u)$.
Since $0\le h,h'\le p$, expanding the four factors gives
\[
 (1-H)(1-H')(1+h)(1+h')
 \le 1+(h-H)+(h'-H')+p^2+(1+p)^2HH'.
\]
The averaging estimate cancels the two linear differences after
multiplication by $F$ and expectation. Thus, writing $M_b=\mathbb E Z_b^2$,
\begin{equation}\label{od:lower-second-step}
 M_b\le\theta^2(1+p^2)M_{b-1}
       +\theta^2(1+p)^2\mathbb E\int_{\T^2}FHH'\,dt\,du
       +O(N^{-10}).
\end{equation}
For any $L$ points, the set of $u$ for which $I_u(s)$ contains at least
one of them has length at most $Ls$. Since $U_{b-1}(x,u)\le1$, this gives
\begin{align*}
 \mathbb E\int_{\T^2}FHH'\,dt\,du
 &\le Ls\,\mathbb E\int_\T U_{b-1}(x,t)H_b(x,t)\,dt\\
 &\le Ls\left(p\,\mathbb E Z_{b-1}+N^{-10}\right)
 \le Cp^2\theta^{b-1}+O(N^{-8}),
\end{align*}
where the second inequality uses the averaging estimate and $h_b\le p$,
and the third uses \eqref{od:lower-first}.
Divide \eqref{od:lower-second-step} by $\theta^{2b}$ and use the
preceding estimate of the cross term. We obtain
\[
 \frac{M_b}{\theta^{2b}}
 \le(1+p^2)\frac{M_{b-1}}{\theta^{2(b-1)}}
    +Cp^2\theta^{1-b}+O(N^{-8}\theta^{-2b}).
\]
Since $Bp^2=o(1)$,
\[
 (1+p^2)^B=1+O(Bp^2),\qquad
 \sum_{b=1}^B\theta^{1-b}
 =\frac{\theta(\theta^{-B}-1)}p\ll\frac{\theta^{-B}}p.
\]
Also, $B\le N$ and $\theta^{-B}=N^{c+o(1)}$; hence the accumulated
$O(N^{-8}\theta^{-2b})$ errors are $O(N^{-5})$. Iteration gives
\[
 \frac{M_B}{\theta^{2B}}
 \le 1+O(Bp^2+p\theta^{-B})+O(N^{-5}).
\]
By \eqref{od:lower-first}--\eqref{od:lower-mean},
$\mathbb EZ_B=\theta^B(1+O(Bp^2)+O(N^{-8}\theta^{-B}))$.
Subtracting $(\mathbb EZ_B)^2$ from $M_B$ and dividing by
$(\mathbb EZ_B)^2$ therefore gives
\[
 \frac{\operatorname{Var}Z_B}{(\mathbb EZ_B)^2}
 \ll Bp^2+p\theta^{-B}+N^{-5}.
\]
By \eqref{od:lower-mean}, $\theta^{-B}=N^c(1+o(1))$, and hence
\[
 \begin{aligned}
 Bp^2&\ll NLs^2\ll N^{\xi-1}(\log N)^2,\\
 p\theta^{-B}&\ll LsN^c
    \ll N^{\xi+c-1}\log N=N^{-\xi}\log N.
 \end{aligned}
\]
Since $\xi=(1-c)/2$, both errors are
$O(N^{-\xi}(\log N)^2)$, and we conclude that
\begin{equation}\label{od:variance}
 \frac{\operatorname{Var}Z_B}{(\mathbb EZ_B)^2}
 =O(N^{-\xi}(\log N)^2).
\end{equation}

It remains to restore the omitted indices. We claim that, for every omitted index $n$,
\begin{equation}\label{od:one-omitted}
 \mathbb E\int_\T U_B(x,t)\,
            \1_{I_t(s)}(a_nx\bmod1)\,dt
 \le Cs\theta^B+O(N^{-8}).
\end{equation}
Put
$J(x,t)=\1_{I_t(s)}(a_nx\bmod1)$.
If $n$ is between blocks $b$ and $b+1$, remove those two block factors
from $U_B$, which increases it. The remaining frequencies on either
side of $a_n$ are separated from it by at least $L$ indices. Since
$q^L$ exceeds every fixed power of $N$, the single-insertion estimate
\eqref{od:single-insertion} gives
\[
 \mathbb E[U_{b-1}(\cdot,t)J(\cdot,t)]
 \le s\,\mathbb E U_{b-1}(\cdot,t)+O(N^{-10}).
\]
For each retained later block $j$, let $F$ be the product already formed,
including $J$. Its variation obeys the same bound because the inserted
indicator has at most $C(1+a_n)$ jumps on a fixed bounded interval.
The separation from the remaining blocks permits the use of
\eqref{od:transfer}. The first-moment identity, with its last term
nonpositive,
therefore gives
\[
 \mathbb E[\theta F(1-H_j)(1+h_j)]
 \le\theta\mathbb E F+O(N^{-10}).
\]
Iterating over the $B-b-1$ retained later blocks proves a bound
$Cs\theta^{B-2}+O(N^{-8})$.
If $n$ is after the last full block, remove just that block and use
\eqref{od:single-insertion}; the bound is
$Cs\theta^{B-1}+O(N^{-8})$.
If $n$ is before the first full block, remove that block and start the
later-block recursion from $J$. Its initial expectation after averaging
in $t$ is exactly $s$, because
$\int_\T\1_{I_t(s)}(y)\,dt=s$ for every $y\in\T$.
The first removed block supplies the required separation, and again the
bound is $Cs\theta^{B-1}+O(N^{-8})$.
Since $\theta^{-2}$ is bounded, all three cases prove
\eqref{od:one-omitted}.

Let $\mathcal O_N=\{1,\ldots,N\}\setminus\bigcup_{b=1}^B\mathcal I_b$
be the omitted indices, and set
\[
 S(x)=\int_\T U_B(x,t)
   \1_{\{a_nx\bmod1\in I_t(s)\text{ for some }n\in\mathcal O_N\}}\,dt.
\]
By the union bound, summing \eqref{od:one-omitted} over the $m_N$
omitted indices gives
\[
 \mathbb E S\le Cm_Ns\theta^B+O(N^{-7}).
\]
If $Z_B>S$, some interval of length $s$ avoids every one of the first
$N$ points. Chebyshev's inequality for $Z_B$ and Markov's inequality
for $S$, with thresholds $\frac12\mathbb E Z_B$ and
$\frac14\mathbb E Z_B$, respectively, show that
\[
 \mathbb P\{G_N<s\}
 \le \frac{4\operatorname{Var}Z_B}{(\mathbb E Z_B)^2}
       +\frac{4\mathbb E S}{\mathbb E Z_B}
 =O(N^{-\xi}(\log N)^2).
\]
This is the required polynomially small lower-tail probability.

\subsection{Almost-sure conclusion}
Fix $r>1$ and $0<c<1<c'<2$, and put $N_j=\lfloor r^j\rfloor$.
The two failure probabilities are summable along this sequence.
By Borel--Cantelli, for Lebesgue-almost every $x\in Q_0$ and all
sufficiently large $j$,
\[
 c\frac{\log N_j}{N_j}\le G_{N_j}(x)
 \le c'\frac{\log N_j}{N_j}
 \,.
\]
If $N_j\le N<N_{j+1}$, monotonicity gives
\[
 \frac{N_jG_{N_{j+1}}(x)}{\log N_{j+1}}
 \le\frac{NG_N(x)}{\log N}
 \le\frac{N_{j+1}G_{N_j}(x)}{\log N_j}.
\]
Thus the liminf is at least $c/r$ and the limsup at most $rc'$.
Let $r\downarrow1$, $c\uparrow1$, and $c'\downarrow1$ through
countable sequences. Finally cover $\R$ by countably many bounded
intervals $Q_0$. This proves the claimed Lebesgue-almost-everywhere limit.

\section{Independent uniform points and geometric approximation}\label{sec:iid-points}
We prove \cref{thm:iid} and establish the geometric estimates used later
for lacunary dilates.

\subsection{Normalization and finite inner approximations}
Let $B_2$ be the closed Euclidean unit ball and
$\mathbf e_1,\ldots,\mathbf e_d$ the coordinate vectors. For matrices,
$\|\cdot\|$ denotes the induced operator norm. Equip the rotation group
$SO(d)$ with normalized Haar measure.

\begin{lemma}\label{ac:normalization}
Every convex body $K\subset[0,1]^d$ can be written
\begin{equation}\label{ac:normalize}
 K=c_K+U\mathsf D\mathcal P,\qquad
 \mathsf D=\operatorname{diag}(\ell_1,\ldots,\ell_d),\quad
 \ell_1\ge\cdots\ge\ell_d>0,\qquad
 B_2\subset\mathcal P\subset d^2B_2,
\end{equation}
where $U\in SO(d)$.  Consequently, for $s=\vol(K)$ and $w=\ell_d$,
\begin{equation}\label{ac:scales}
 \ell_1\le C_d,\qquad
 c_ds\le\prod_{j=1}^d\ell_j\le C_ds,\qquad
 c_ds\le w\le C_ds^{1/d}.
\end{equation}
For each $0<\varepsilon<1$, one can choose a finite family of convex
polytopes and a number $\tau>0$, depending only on $d,\varepsilon$, such
that for each normalized body $\mathcal P$ in \eqref{ac:normalize}, a
polytope $P$ in this family satisfies
\[
 P+\tau B_2\subset\mathcal P,\qquad
 \vol(P)\ge(1-\varepsilon)\vol(\mathcal P),\qquad
 \tfrac12 B_2\subset P\subset d^2B_2.
\]
\end{lemma}
\begin{proof}
Let $T=\operatorname{conv}\{v_0,\ldots,v_d\}\subset K$ be a
maximal-volume simplex and let $c_K$ be its centroid.  If
$x=\sum_{j=0}^d\lambda_jv_j\in K$, with $\sum_j\lambda_j=1$, then
replacing $v_j$ by $x$ produces a simplex in $K$ of volume
$|\lambda_j|\vol(T)$.  Maximality gives $|\lambda_j|\le1$.
The numbers $(1-\lambda_j)/d$ are nonnegative and sum to
one, so
\[
 x=(d+1)c_K-d\sum_{j=0}^d\frac{1-\lambda_j}{d}v_j
 \in c_K-d(T-c_K).
\]
Thus $T\subset K\subset c_K-d(T-c_K)$; see also
Lagarias--Ziegler~\cite[Theorem~3(a)]{LagariasZiegler1991}.
Let $\Delta$ be a regular simplex centered at the origin with inradius
one, so $B_2\subset\Delta\subset dB_2$, and write
$T=c_K+L\Delta$.  It follows that
\[
 B_2\subset L^{-1}(K-c_K)\subset-d\Delta\subset d^2B_2.
\]
A singular value decomposition of $L$ gives \eqref{ac:normalize}; a
reflection may be absorbed into $\mathcal P$ so that $U\in SO(d)$.
The inclusion $c_K+U\mathsf D B_2\subset[0,1]^d$ bounds $\ell_1$.
Comparing volumes gives
$\prod_j\ell_j\asymp_ds$, and the ordering of the scales gives the bounds
for $w$ in \eqref{ac:scales}.

For the polytope approximation, choose small $\rho,h>0$,
and let $P$ be the convex hull of the points of $h\Z^d$ in
$(1-\rho)\mathcal P$.  Put $R=d^2$.  Since
$B_2\subset\mathcal P\subset RB_2$, every point of
$(1-\rho-\sqrt d h)\mathcal P$ has a ball of radius $\sqrt d h$
inside $(1-\rho)\mathcal P$.  The lattice cube containing that point
therefore has all its vertices in $(1-\rho)\mathcal P$, and hence
\[
 (1-\rho-\sqrt d h)\mathcal P\subset P\subset(1-\rho)\mathcal P.
\]
Also $P+\rho B_2\subset\mathcal P$.  Choosing $\rho,h$ sufficiently
small proves the volume and ball inclusions, with $\tau=\rho$.
Only finitely many lattice points occur in $RB_2$, so the possible
polytopes $P$ form a finite family.
\end{proof}

\begin{lemma}\label{ac:catalogue}
For every fixed $0<\varepsilon<1$ and all sufficiently small $s>0$,
there is a collection $\mathscr K_s$ of convex bodies in $[0,1]^d$
such that every convex body in the cube of volume at least $s$ contains
a member. Every member has volume exactly $(1-\varepsilon)s$, and
\begin{equation}\label{ac:entropy}
 \#\mathscr K_s\le C_{d,\varepsilon}s^{-d}
                         (\log(1/s))^{C_d}.
\end{equation}
The mesh center, scales, and orthogonal frame associated with each
member are retained after the final volume adjustment.
\end{lemma}
\begin{proof}
It suffices to approximate bodies of volume exactly $s$, since every
larger body contains such a homothetic subbody.  Apply
\cref{ac:normalization}.  Choose geometric meshes of ratio $1+\delta$
for the scales $\ell_j$, where $\delta>0$ will depend only on
$d,\varepsilon$.  By \eqref{ac:scales}, there are at most a fixed power
of $\log(1/s)$ scale choices.  We keep only choices with
$\det\mathsf D\asymp_ds$.  For each choice of scales and orientation,
a center mesh with spacings $\delta\ell_j$ in the frame directions has
$O_{d,\delta}((\det\mathsf D)^{-1})=O_{d,\delta}(s^{-1})$ points in a
fixed bounded region containing the cube.

For the orientation net, set
\[
 \mathcal U_\gamma(\mathsf D)
 =\{R\in SO(d):\|\mathsf D^{-1}(R-I)\mathsf D\|\le\gamma\}.
\]
For fixed sufficiently small $\gamma>0$,
\begin{equation}\label{eq:orientation-neighborhood}
 \operatorname{Haar}(\mathcal U_\gamma(\mathsf D))
 \ge c_{d,\gamma}\prod_{j<k}\frac{\ell_k}{\ell_j}.
\end{equation}
Indeed, take skew-symmetric matrices $S$ with
$|S_{jk}|\le c_d\gamma\ell_k/\ell_j$ for $j<k$.  Then
$\|\mathsf D^{-1}S\mathsf D\|\le\gamma/2$ if $c_d$ is small enough,
so $e^S\in\mathcal U_\gamma(\mathsf D)$. These matrices have uniformly
small operator norm, so the exponential map has a Jacobian bounded away
from zero on this box. The box volume gives
\eqref{eq:orientation-neighborhood}.

For $\gamma$ sufficiently small compared with $\delta$, choose a maximal
family of disjoint left translates
$U_\alpha\mathcal U_\gamma(\mathsf D)$.  If
$R_1,R_2\in\mathcal U_\gamma(\mathsf D)$, then
\[
 \|\mathsf D^{-1}(R_1R_2^{-1}-I)\mathsf D\|
 \le \frac{2\gamma}{1-\gamma}\le\delta,
\]
where the first inequality follows by writing
$\mathsf D^{-1}R_i\mathsf D=I+E_i$ and using
$\|(I+E_2)^{-1}\|\le(1-\gamma)^{-1}$.
Consequently, maximality ensures that every $U\in SO(d)$ satisfies
$U_\alpha^{-1}U\in\mathcal U_\delta(\mathsf D)$ for some $\alpha$.
The number of chosen frames is therefore at most
\[
 C_{d,\delta}\prod_{j<k}\frac{\ell_j}{\ell_k}
 \le C_{d,\delta}(\det\mathsf D)^{-(d-1)}
 \ll_{d,\delta}s^{-(d-1)}.
\]
For the middle inequality, multiply by $(\det\mathsf D)^{d-1}$ and use
\[
 (\det\mathsf D)^{d-1}\prod_{j<k}\frac{\ell_j}{\ell_k}
 =\prod_{j=1}^d\ell_j^{2(d-j)}\ll_d1.
\]
When $d=1$, the orientation group is a singleton and the same formulas
hold with empty products equal to one.

Use the finite normalized polytopes from \cref{ac:normalization}, with
a smaller preliminary volume loss than $\varepsilon$. In coordinates
normalized by $(U\mathsf D)^{-1}$, the center, scale, and orientation
mesh errors move each point of an approximant by at most $C_d\delta$.
The margin $P+\tau B_2\subset\mathcal P$ absorbs these
errors if $\delta$ is sufficiently small.  The approximating body then
lies inside $K$ and has volume at least $(1-\varepsilon)s$.
Discard representatives not contained in the cube or not satisfying this
volume bound, and contract each remaining representative about its mesh
center to volume exactly $(1-\varepsilon)s$. Since the normalized
polytope contains $\tfrac12B_2$, the mesh center lies in the
representative, so contraction preserves containment. Keep the mesh
scales and frame after contraction; the normalized bodies remain
between two fixed concentric balls. Multiplying the orientation and
center counts and the logarithmic number of scale choices proves
\eqref{ac:entropy}.
\end{proof}

\subsection{Thin targets and intersections}

\begin{lemma}\label{lem:two-slab-intersection}
Let $d\ge2$, let $E\subset\R^d$ be measurable with diameter at most
$D$, and let $a,b\in\R^d$ be linearly independent.  If $I,J$ are
intervals of lengths $w,w'$, then
\[
 \vol\{x\in E:a\cdot x\in I,\ b\cdot x\in J\}
 \le C_dD^{d-2}\frac{ww'}
 {\sqrt{\|a\|^2\|b\|^2-(a\cdot b)^2}}.
\]
\end{lemma}
\begin{proof}
Slice along affine planes parallel to $\operatorname{span}\{a,b\}$.
On each plane, the map $x\mapsto(a\cdot x,b\cdot x)$ has
absolute Jacobian
$\sqrt{\|a\|^2\|b\|^2-(a\cdot b)^2}$, so the two strip conditions
have area $ww'$ divided by this number. The projection of $E$ onto the
orthogonal complement lies in a ball of radius $D$, of
$(d-2)$-dimensional volume at most $C_dD^{d-2}$. Integrating the area
bound over this ball proves the claim; for $d=2$ the latter factor is one.
\end{proof}

For both the independent and lacunary lower bounds, put $I_0=[1/4,1/2]$
and, for $d\ge2$,
\[
 U_0=[0,1/(8d)]^{d-1},\qquad \Theta=U_0\times I_0,
\]
with normalized Lebesgue measure $\nu$.  For $\vartheta=(u,t)\in\Theta$,
set
\begin{equation}\label{eq:model-slabs}
 S_\vartheta(w)=\{(z,y):z\in(0,1)^{d-1},\quad
                         t+u\cdot z<y<t+u\cdot z+w\}.
\end{equation}
For small $w>0$, this is the interior of a convex body in $[0,1]^d$
of volume $w$. In dimension one, use $\Theta=I_0$, $S_t(w)=(t,t+w)$,
and let $\nu$ be normalized Lebesgue measure on $I_0$.

\begin{lemma}\label{lem:iid-thin-boxes}
Let $\Omega\subset\R^d$ be a convex body.  There are constants
$C,\delta_0>0$, depending only on $\Omega$, such that, whenever
$0<Cs\le\delta\le\delta_0$, there are $M\asymp_\Omega\delta^{-d}$
convex bodies $Q_1,\ldots,Q_M$ compactly contained in
$\operatorname{int}\Omega$, each of volume $s$, with
\begin{equation}\label{eq:iid-box-intersection}
 \vol(Q_j\cap Q_k)\le C_\Omega s^2/\delta
 \qquad(j\ne k).
\end{equation}
In dimension one the intervals can be chosen pairwise disjoint.
\end{lemma}
\begin{proof}
Choose a cube $x_0+r[0,1]^d$ compactly contained in
$\operatorname{int}\Omega$.  Put $w=s/r^d$, choose a Cartesian grid of
spacing $\delta$ in $\Theta$, and take the closures of
$x_0+rS_\vartheta(w)$ over this grid.
For $C$ large and $\delta_0$ small, $w<\delta$ and all these targets
lie in the cube.  Their number is $\asymp_\Omega\delta^{-d}$, and
each has volume $r^dw=s$.

The intervals in dimension one, and targets with equal slopes in higher
dimensions, are disjoint because their intercepts are separated by at
least $\delta>w$. For distinct slopes, some coordinate differs by at
least $\delta$. For their normals
$a=(-u,1)$ and $b=(-u',1)$,
\[
 \sqrt{\|a\|^2\|b\|^2-(a\cdot b)^2}
 \ge\|u-u'\|\ge\delta.
\]
The first inequality follows by retaining the minors involving the last
coordinate in the Gram determinant. Apply \cref{lem:two-slab-intersection}
in the unit cube, with both strip widths $w$, and then scale by $r$.
The resulting intersection
bound is $C_dr^dw^2/\delta=C_\Omega s^2/\delta$.
\end{proof}

\subsection{The independent-point theorem}
Under the hypotheses of \cref{thm:iid}, for every $\varepsilon>0$ there
are $\alpha>0$ and $C_{\Omega,\varepsilon}>0$ such that
\begin{equation}\label{eq:iid-tails}
 \mathbb P\left(\left|\frac{N V_N^*}{\vol(\Omega)\log N}-d\right|
                >\varepsilon\right)
 \le C_{\Omega,\varepsilon}N^{-\alpha}\qquad(N\ge2).
\end{equation}
\begin{proof}[Proof of \cref{thm:iid} and \eqref{eq:iid-tails}]
By rescaling, we may assume $\vol(\Omega)=1$. The random variable $V_N^*$
is Borel measurable: the supremum may be restricted to the countable
family of rational polytopes. Contract an admissible body slightly about an interior point,
then approximate it from inside by a rational polytope, with arbitrarily
small total volume loss. The contracted body avoids all sample points,
so requiring the whole body to avoid the sample gives the same supremum.

For the upper bound, fix $c>d$, choose $\varepsilon_0>0$ with
$(1-\varepsilon_0)c>d$, and put $s=c\log N/N$.  Enclose $\Omega$ in a
fixed cube.  Rescaling \cref{ac:catalogue} to this cube and discarding
members not contained in $\Omega$ gives an inner approximating family
of size at most
\[
 C_{\Omega,\varepsilon_0}s^{-d}(\log(1/s))^{C_d},
\]
whose members have volume $(1-\varepsilon_0)s$. Independence and a
union bound give
\begin{align*}
 \mathbb P(V_N^*>s)
 &\le C_{\Omega,\varepsilon_0}s^{-d}(\log(1/s))^{C_d}
             (1-(1-\varepsilon_0)s)^N\\
 &\le N^{d-(1-\varepsilon_0)c+o(1)}.
\end{align*}

For the lower bound, fix $0<c<d$, choose $c/d<\beta<1$, and put
\[
 s=c\log N/N,\qquad \delta=N^{-\beta}.
\]
For all large $N$, \cref{lem:iid-thin-boxes} gives
$M\asymp_\Omega\delta^{-d}$ convex targets of volume $s$ satisfying
\eqref{eq:iid-box-intersection}.  Let $Z$ count those with empty
interior, and put $p=(1-s)^N$.  Then
\[
 \mathbb EZ=Mp\asymp_\Omega N^{\beta d-c}.
\]
For two distinct targets with overlap volume $I$, independence gives
\[
 \frac{\mathbb P(\text{both targets are empty})}{p^2}
 =\left(\frac{1-2s+I}{(1-s)^2}\right)^N
 \le\exp(C_\Omega Ns^2/\delta).
\]
Since $Ns^2/\delta\to0$, the variance bound is
\begin{align*}
 \frac{\operatorname{Var}Z}{(\mathbb EZ)^2}
 &\le\frac1{Mp}+\exp(C_\Omega Ns^2/\delta)-1\\
 &\ll_\Omega N^{c-\beta d}+N^{\beta-1}(\log N)^2.
\end{align*}
The choice $\beta=(c+1)/(d+1)$ lies in $(c/d,1)$ and makes both powers
of $N$ equal to $-(d-c)/(d+1)$. By Chebyshev's inequality,
$\mathbb P(Z=0)$ is bounded by a negative power of $N$, and
\[
 \mathbb P(V_N^*<c\log N/N)\le\mathbb P(Z=0).
\]

Taking $c=d+\varepsilon$ above, and $c=d-\varepsilon$ below when
$0<\varepsilon<d$, proves \eqref{eq:iid-tails}.
For $\varepsilon\ge d$ the lower-tail events are empty.  Increasing the
constant handles the finitely many small values of $N$.

The geometric-subsequence argument in Section~\ref{sec:one-dimensional}
applies to \eqref{eq:iid-tails}: $V_N^*$ is nonincreasing, and the tails
are summable along $N_m=\lceil r^m\rceil$ for each $r>1$.
Let $r\downarrow1$ and $\varepsilon\downarrow0$ through countable
sequences. Undoing the homothety yields
\[
 \frac{NV_N^*}{\vol(\Omega)\log N}
 \longrightarrow d\qquad\text{almost surely}.\qedhere
\]
\end{proof}

\section{Smoothing and transfer for Fourier-decaying measures}\label{sec:prelim}
We extend \cref{lem:od-transfer} to Fourier-decaying measures by
polynomial approximation of block indicators and preceding weights.

\subsection{Notation and target boundaries}
Implied constants may depend on $d,q,a_1,C_\mu,\eta$.
In the fixed-body case, they may also depend on $\mathcal C$.
By translating and rescaling $\mathcal C$, we may assume
$0\in\operatorname{int}\mathcal C$ and $\vol(\mathcal C)=1$, adjusting
the scale cutoff.  Write $\lambda_d$ for Haar probability on $\T^d$.
Unless specified otherwise,
expectations and $L^p$ norms on $\R^d$ use $\mu$; torus integrals use
$\lambda_d$.  Periodic functions are evaluated modulo
$\Z^d$.  Euclidean and maximum-norm boundary estimates
are equivalent up to dimensional constants.  Lacunarity gives
\begin{equation}\label{eq:growth}
 a_j/a_k\le q^{j-k}\ (j\le k),\qquad
 \sum_{j\le k}a_j\ll_q a_k,\qquad a_k\ge a_1q^{k-1}.
\end{equation}
The suprema in the introduction are measurable: strict superlevel sets
are countable unions of fixed-target avoidance events, using inner
rational polytopes for arbitrary convex bodies, and a slight contraction
followed by rational translation and scale for homothets.

The targets and incidence sets used below arise by projection modulo
$\Z^d$ from unions of a uniformly bounded number of convex sets in
a fixed bounded box. They satisfy
\begin{equation}\label{eq:uniform-boundary}
 \lambda_d\{y:\dist(y,\partial J)\le\delta\}\ll\delta
 \qquad(0<\delta\le1),
\end{equation}
uniformly in their eccentricity.  The boundary layer in
\eqref{eq:uniform-boundary} is taken after projection to $\T^d$.
For a convex lift $K$ in a box of bounded side lengths, let $\pi_j$ denote the
coordinate projection that deletes coordinate $j$, and successively add
$[-\delta,\delta]\mathbf e_j$.  The volume added in direction $j$ is
$2\delta\vol_{d-1}(\pi_jK')=O(\delta)$, where $K'$ is the current
convex set.  This bounds the outer layer.  For an inner-layer point
$y$, a nearest supporting hyperplane has unit normal $n$ and distance
at most $\delta$.  Some $|n_j|\ge d^{-1/2}$, so the coordinate fiber
through $y$ ends within $\sqrt d\delta$ in one of its two directions.
The union of these coordinate endpoint layers has volume $O(\delta)$
by integration over the coordinate projections.  Projection and bounded
unions preserve the estimate.  For
$\mathcal C_s=s^{1/d}\mathcal C$ and
$J_t=t+\operatorname{int}\mathcal C_s\pmod{\Z^d}$ we also have
\begin{equation}\label{eq:boundary}
 \vol\{y:\dist(y,\partial\mathcal C_s)\le\delta\}
 \ll_{\mathcal C}s^{(d-1)/d}\delta
 \quad(0<\delta\ll_{\mathcal C}s^{1/d}).
\end{equation}
The layer lies between homothets with scales
$s^{1/d}\pm O_{\mathcal C}(\delta)$, since $\mathcal C$ contains a
ball about $0$.

\begin{lemma}\label{lem:interval}
Put $\eta_0=\tfrac12\min\{\eta,1\}$.  If $I\subset\T$ is a union of
at most a fixed number of arcs of total length $\delta\le1/2$, then
\begin{equation}\label{eq:fourier-interval}
 \mu\{x:Dx_j\bmod1\in I\}
 =\delta+O(D^{-\eta}\delta^{\eta_0})\qquad(D\ge a_1).
\end{equation}
\end{lemma}
\begin{proof}
The nonconstant Fourier terms are bounded by
$CD^{-\eta}\sum_{h\ge1}\min\{\delta,h^{-1}\}h^{-\eta}
\ll D^{-\eta}\delta^{\eta_0}$.
The same estimate for smooth arc majorants, followed by shrinking the
arc, shows that the projected measure has no atoms.  Fej\'er
approximation justifies the calculation for the arc indicators.
\end{proof}

\subsection{Local averages}
For the rest of the paper, choose a nonnegative radial kernel
$\Phi\in C_c^\infty(\R^d)$ with integral one, supported in $B_2$
and nonincreasing in the radius.  Put
\[
 \Phi_T(u)=T^d\Phi(Tu),
 \qquad \ell=\log N,\qquad \Lambda=e^{\ell^6}.
\]
Convolution with $\Phi_T$ is a positive mixture of normalized ball
averages with radii at most $T^{-1}$; the mixture mass on radii below
$r$ is $O((Tr)^d)$, by integration of the radial derivative of $\Phi$.
Its Fourier multiplier satisfies
$|1-\widehat\Phi(\xi)|\ll\min\{1,\|\xi\|^2\}$.
For any measurable periodic $0\le f\le1$ of Haar mean $v$,
\begin{equation}\label{eq:single-spatial-marginal}
 \int_{\R^d}\Phi_T(x-y)f(ay+\theta)\,dy=v+O(T/a)
 \qquad(a>0),
\end{equation}
uniformly in $x,\theta$.  To prove this, write the integral with kernel
$G(y)=\Phi_T(x-y)$, translate $y$ by $u$, and replace $G(y+u)$ by
$G(y)$ at cost at most
$\|\nabla G\|_{L^1(dy)}\|u\|\ll T\|u\|$.
Averaging $u$ over $[0,a^{-1}]^d$ replaces the periodic factor by $v$.

For a finite exponential sum $P=\sum_\xi c_\xi e^{2\pi i\xi\cdot x}$,
with real vector frequencies, collect equal frequencies and put
$\|P\|_{\mathcal W}=\sum_\xi|c_\xi|$ and
$\operatorname{spec}P=\{\xi:c_\xi\ne0\}$.  The Wiener norm is
submultiplicative, and the spectrum of a product is contained in the
sum of the spectra of its factors.

\subsection{Trigonometric approximation of avoidance events}
Throughout this subsection targets have volume $O(s)$, where
$s\asymp\ell/N$, satisfy
\eqref{eq:uniform-boundary}, and may vary with the index and carry
arbitrary prescribed phase shifts.

\begin{lemma}\label{lem:smoothing}
For every $B\ge2$, each target $J$ admits a trigonometric polynomial
$\psi_J$ with $0\le\psi_J\le2$ and mean $\lambda_d(J)+O(N^{-B})$.
Its frequency radius and Wiener norm are at most $N^{C_B}$, and
it satisfies
\begin{equation}\label{eq:smooth-measure}
 \int|\psi_J(ax+\theta)-\1_J(ax+\theta)|\,d\mu(x)
 \ll N^{-B}+|a|^{-\eta}N^{C_B}\qquad(a\ne0).
\end{equation}
For fixed-body translates these polynomials may be chosen so that
\begin{equation}\label{eq:smooth-target}
 \int_{\T^d}|\psi_t(y)-\1_{J_t}(y)|\,dt\ll N^{-B}
 \qquad(y\in\T^d).
\end{equation}
\end{lemma}
\begin{proof}
Convolve $\1_J$ with a smooth probability kernel of radius
$\delta=N^{-B}$, truncate its Fourier series with uniform error at
most $\delta$, and add $\delta$ to ensure positivity.  Derivative
bounds for the mollifier give both stated polynomial bounds by choosing
a sufficiently large polynomial cutoff.  The pointwise error is at
most $2\delta$ outside the $\delta$-boundary layer and at most two
inside it.  A smooth majorant of that layer has mean $O(\delta)$,
by \eqref{eq:uniform-boundary}, and polynomial Wiener norm.  Its
nonzero modes, estimated by Fourier decay, prove
\eqref{eq:smooth-measure}.  Translating one fixed polynomial and
integrating in $t$ prove \eqref{eq:smooth-target}.
\end{proof}

\begin{proposition}\label{prop:avoidance}
For every $A>0$ there is $C_*>0$ such that, for
$n_* =\lceil C_*\log N\rceil$ and any
$\mathcal I\subset\{n_*,\ldots,N\}$, the avoidance indicator
$R=\prod_{n\in\mathcal I}(1-\1_{J_n}(a_nx+\theta_n))$
has a finite exponential-sum approximation $P_R$ with
\begin{equation}\label{eq:avoidance-approx}
 \|R-P_R\|_2\le N^{-A},\qquad
 \|P_R\|_{\mathcal W}\le e^{C_A\ell^4},\qquad
 \operatorname{spec}P_R\subset\{\|\xi\|\le N^{C_A}a_m\},
\end{equation}
where $m=\max(\mathcal I\cup\{1\})$.  If $H=1-R$ uses only frequencies $a_n\ge a_-$ and $a_-/T\ge\Lambda$,
the same construction can be made to satisfy, for $P_H=1-P_R$,
\begin{equation}\label{ac:convolution-approx}
 \|\Phi_T*(P_H-H)\|_\infty\le N^{-A}.
\end{equation}
By inclusion--exclusion, the same bounds hold for unions of two
targets and products of two hit indicators.
\end{proposition}
\begin{proof}
If $\mathcal I=\varnothing$, take $P_R=1$.
Otherwise, for each $n\in\mathcal I$ let $\psi_n=\psi_{J_n}$
be supplied by \cref{lem:smoothing}, and set
$X=\sum_{n\in\mathcal I}\1_{J_n}(a_nx+\theta_n)$ and
$Z=\sum_{n\in\mathcal I}\psi_n(a_nx+\theta_n)$. Fix $D>2A+4$.
Choose $B$ and then $C_*$ sufficiently large that
\cref{lem:smoothing} and \eqref{eq:growth} give
$\mathbb E_\mu|X-Z|\le N^{-D}$.
We first prove
\begin{equation}\label{eq:exponential}
 \mathbb E_\mu e^{\sigma Z}\le C,
 \qquad\sigma\asymp\ell^{-1}.
\end{equation}
Choose fixed $t_0,b_0>0$ with $e^{t_0\psi_n}\le1+b_0\psi_n$.
Split the indices into $r\asymp\ell$ residue classes so that successive
frequencies in a class have ratio at least $Q$, where
$Q>2N^{C_B}+1$ and $Q^{-\eta}(1+b_0N^{C_B})<1/2$.
In a nonempty class write its frequencies as $v_1<\cdots<v_M$.
In a Fourier product whose last nonzero mode occurs at $j$, the
frequency norm is at least $v_j/2$: the last integer mode has norm
at least one and the preceding modes have sum smaller than $v_j/2$.
If $p=\max_n\int\psi_n=O(s)$, all nonconstant contributions in the
expanded product have absolute sum at most
\[
 C v_1^{-\eta}b_0N^{C_B}(1+b_0p)^M
 \sum_{j\ge1}[Q^{-\eta}(1+b_0N^{C_B})]^{j-1}
 \le (1+b_0p)^M,
\]
after increasing $C_*$.  The constant contribution is at most the
same quantity.  H\"older over the classes, with
$\sigma=t_0/r$, gives
$\mathbb E_\mu e^{\sigma Z}\le2e^{CNp/r}\le C$.
The bounds are uniform in the phases, since absolute values were taken
before collecting coincident frequencies.

Let $k=\lceil K_A\ell^3\rceil$, $M=16k/\sigma$, and
$c=(M+1/2)/(M-1/2)$.  In terms of the Chebyshev polynomial
$T_k(\cos u)=\cos(ku)$, put
\[
 P(z)=\frac{T_k(c-2z/(M-1/2))}{T_k(c)}.
\]
The identities $T_k(\cosh u)=\cosh(ku)$ and
$T_{k+1}(u)=2uT_k(u)-T_{k-1}(u)$ give
\[
 P(0)=1,\quad 0\le P(z)\le1\ (0\le z\le1/2),\quad
 |P(z)|\le2e^{-k/\sqrt M}\ (1/2\le z\le M).
\]
On $[0,1/2]$, differentiation gives
$|1-P(z)|\le4k^2z/M$, since $T_k'(u)/T_k(u)\le k^2$ for $u\ge1$.
For $z>M$, $|P(z)|^2\le(8z/M)^{2k}\le e^{\sigma z/2}$.
Consequently, splitting these three ranges and using that $X$ is a
nonnegative integer, we obtain
\[
 \mathbb E_\mu|\1_{\{X=0\}}-P(Z)|^2
 \ll (1+k^2/M)\mathbb E_\mu|X-Z|
       +e^{-2k/\sqrt M}+e^{-\sigma M/2}.
\]
On $Z\le1/2$, the cases $X=0$ and $X\ge1$ are bounded
by $4k^2|X-Z|/M$ and $2|X-Z|$, respectively; for $Z>M$, use
\eqref{eq:exponential}.  Since $k^2/M=O_A(\ell^2)$ and
$k/\sqrt M\asymp\sqrt{k\sigma}$, choosing $K_A$ sufficiently
large gives the required $L^2$ accuracy.
The recurrence for $T_k$ also gives
$\|P(Z)\|_{\mathcal W}\le[C(1+\|Z\|_{\mathcal W}/M)]^k
\le e^{C_A\ell^4}$.  Its frequencies are sums of at most $k$
frequencies of $Z$, proving the spectral bound.

For \eqref{ac:convolution-approx}, fix a center $x$ and use the
probability measure $\Phi_T(x-y)\,dy$ in place of $\mu$.  Its Fourier
transform has modulus $|\widehat\Phi(\xi/T)|$, and
$|\widehat\Phi(\zeta)|\ll_M(1+\|\zeta\|)^{-M}$ for every $M>0$.
In each separated residue class, a nonzero sum of Fourier modes has
norm at least half the largest participating frequency, hence at least
$a_-/2\ge T\Lambda/2$.  The smoothing-error and exponential-moment
estimates above therefore hold for these local measures, uniformly
in the center and phases, with arbitrary inverse powers of $\Lambda$
in place of Fourier-decay errors.  Choose the smoothing and cutoff
parameters for both estimates together.  The same polynomial then has
$L^1(\Phi_T(x-\cdot))$ error at most $N^{-A}$ for every $x$, proving
\eqref{ac:convolution-approx}.  Convolution multiplies Fourier
coefficients by $\widehat\Phi(\xi/T)$ and does not increase the Wiener norm.
\end{proof}

\subsection{Weights, transfer, and inserted hits}
As in Section~\ref{sec:one-dimensional}, we use products of compensated
survival factors $(1-H_b)(1+h_b)$.  The transfer estimate must also
apply to products for two targets, as needed in the lower-bound second
moment, and to products containing an inserted incidence indicator.
Fix a target $J\subset\T^d$ of Haar measure $s$.
Let the active blocks $\mathcal I_b$ have lengths $L_b$, with
$g=\lceil C_{\rm gap}\log\Lambda\rceil$ omitted indices before the
first and between successive blocks.  Put
\[
 a_b^-=\min_{n\in\mathcal I_b}a_n,\quad
 a_b^+=\max_{n\in\mathcal I_b}a_n,\quad a_0^+=1,\quad
 T_b=a_{b-1}^+\Lambda^2.
\]
\[
 H_b=\1_{\{\exists n\in\mathcal I_b:a_nx\bmod1\in J\}},\qquad
 h_b=\Phi_{T_b}*H_b.
\]
Choose $C_{\rm gap}\log((q+1)/2)>4$.  Then $a_b^-/T_b\ge\Lambda$
for all large $N$, including after the slight frequency shifts used
below.  For equal-length blocks, write $L_b=L$ and omit the last
incomplete block.
By \eqref{eq:single-spatial-marginal} and the union bound,
\begin{equation}\label{ac:smoothed-marginal}
 0\le h_b\le p_b:=L_b(s+e_N),\qquad e_N=O(\Lambda^{-1}).
\end{equation}
We always use blocks for which $p_b\le p_*<1/2$ and
$\sum_bp_b=O(\ell)$.

\begin{lemma}\label{ac:slowweights}
Let $J_1,J_2\subset\T^d$ be targets of Haar measure at most $s$,
satisfying \eqref{eq:uniform-boundary} with a uniform constant.
Let $\mathcal B_1,\mathcal B_2$ be
sets of completed block indices, either of which may be empty.
For $i=1,2$, write $H_{j,i}$ for the hit indicator of $J_i$ in block $j$
and $h_{j,i}=\Phi_{T_j}*H_{j,i}$.  Set
\[
 R=\prod_{i=1}^2\prod_{j\in\mathcal B_i}(1-H_{j,i}),
 \qquad V=\prod_{i=1}^2\prod_{j\in\mathcal B_i}(1+h_{j,i}),
 \qquad F=cRV,\quad |c|\le N^C.
\]
Let
\[
 a_* = \max\bigl(\{1\}\cup
       \{a_n:n\in\mathcal I_j,\ j\in\mathcal B_1\cup\mathcal B_2\}\bigr).
\]
As in Section~\ref{sec:one-dimensional}, $a_*$ records the frequencies
of the preceding block factors; without inserted indicators, it equals $1$
for an empty block product. We have $\|F\|_\infty\le N^{C'}$, and for every $A>0$
there is a finite exponential sum $P_F$ with
\begin{equation}\label{ac:slowweightbounds}
 \|F-P_F\|_2\le N^{-A},\qquad
 \|P_F\|_{\mathcal W}\le e^{C_A\ell^5},\qquad
 \operatorname{spec}P_F\subset\{\|\xi\|\le N^{C_A}a_*\}.
\end{equation}
The same assertions hold after inserting a bounded number of indicators
$\1_E(Dx+\theta)$, each having bounded trigonometric approximations
of arbitrary inverse-power $L^1(\mu)$ accuracy, polynomial Wiener norm,
and spectral radius at most $N^{C_A}|D|$. In this case, enlarge $a_*$
by the absolute frequencies of the indicators already included in $F$;
a new indicator in \eqref{ac:periodic-insertion} is not yet included.
If the input approximation errors are controlled only after integrating
their squared $L^2(\mu)$ norms over target parameters, the $L^2$ bound
in \eqref{ac:slowweightbounds} is understood after the same integration.
The coefficient, spectral, and sup-norm bounds remain uniform.
The conclusions hold uniformly for any choice of averaging scales
$T_j>0$ satisfying $a_j^-/T_j\ge\Lambda$ for all blocks involved,
provided the above bounds on $p_j$ hold.
\end{lemma}
\begin{proof}
The local convolution approximation in \cref{prop:avoidance} requires
only $a_j^-/T_j\ge\Lambda$, and the one-point marginal estimate is
uniform under the same condition.
First $\|V\|_\infty\le\prod_j(1+p_j)^2\le N^{C'}$.
For each pair $(j,i)$, let $u_{j,i}=\Phi_{T_j}*P_{H_{j,i}}$.
By \cref{prop:avoidance},
$\|u_{j,i}-h_{j,i}\|_\infty\le N^{-D}$,
$\|u_{j,i}\|_{\mathcal W}\le e^{C_D\ell^4}$, and
$\operatorname{spec}u_{j,i}\subset\{\|\xi\|\le N^{C_D}a_j^+\}$.
Enumerate these pairs and write $h_1,\ldots,h_B$ and
$u_1,\ldots,u_B$ for the corresponding functions.  For
$\alpha\in\{0,1\}^B$, write
$|\alpha|=\sum_j\alpha_j$ and $z^\alpha=\prod_jz_j^{\alpha_j}$, and put
\[
 Q_K(z)=\sum_{\alpha\in\{0,1\}^{B}:\,|\alpha|\le K}z^\alpha.
\]
The positive product expansion gives
\[
 0\le V-Q_K(h)\le2^{-K-1}\prod_j(1+2h_j)
 \le2^{-K-1}\exp(2\sum_jh_j)\le2^{-K-1}N^{C'}.
\]
Take $K=C_A\lceil\ell\rceil$.  On $|z_j|\le h_j+N^{-D}$ each
partial derivative of $Q_K$ is bounded by
$\prod_l(1+h_l+N^{-D})\le N^{C''}$.
Thus $P_V=Q_K(u)$ approximates $V$ uniformly to accuracy $N^{-A}$
if $D$ is large, and
\[
 \|P_V\|_{\mathcal W}\le(1+\sum_j\|u_j\|_{\mathcal W})^K
 \le e^{C_A\ell^5}.
\]
Its spectral radius is at most $KN^{C_D}a_*$, hence at most a fixed
power of $N$ times $a_*$.
Use \cref{prop:avoidance} for $R$, with the union target for two copies,
and set $P_F=cP_RP_V$.  The $L^2$ error follows by replacing $R$
while keeping the actual bounded $V$, then replacing $V$ uniformly;
$\|P_R\|_2\le2$.  Products give the other two bounds.
For an inserted indicator, boundedness converts arbitrary $L^1$
accuracy into arbitrary $L^2$ accuracy.  Replace that indicator while
retaining the actual bounded weight, then multiply its bounded
polynomial approximation into $P_F$.  Integrating these squared
estimates proves the parameter-integrated version.
\end{proof}

\begin{lemma}\label{ac:transfer}
For a preceding weight $F$ as in \cref{ac:slowweights}, let $H$ be the
hit indicator for the current block, or the product of two such
indicators, and put
$h=\Phi_T*H$.  If $T\ge a_*\Lambda^2$ and the first current
frequency is at least $T\Lambda$, then, for every $A>0$,
\begin{equation}\label{ac:transferidentity}
 \int FH\,d\mu=\int Fh\,d\mu+O_A(N^{-A}).
\end{equation}
For an additional target $E$ satisfying \eqref{eq:uniform-boundary},
with $\lambda_d(E)=v=O(s)$ and
$|D|\ge a_*\Lambda^2$, we also have
\begin{equation}\label{ac:periodic-insertion}
 \int F(x)\1_E(Dx+\theta)\,d\mu
 =v\int F\,d\mu+O_A(N^{-A}).
\end{equation}
The estimates also hold with the inserted indicators allowed in
\cref{ac:slowweights}.  If their approximation errors are integrated
over target parameters, the identities are integrated over the same
parameters.
\end{lemma}
\begin{proof}
Fix the required exponent $A$.  Choose $Q=P_F$ in
\cref{ac:slowweights} with sufficiently high $L^2$ accuracy; then
$\|Q\|_2\le N^{C'}$.  Next choose the current-hit approximant $P$
in \cref{prop:avoidance} so accurately that
$\|H-P\|_2+\|h-\Phi_T*P\|_\infty\le N^{-A-C'-3}$.
Cauchy--Schwarz and the bounds $0\le H,h\le1$ reduce the first claim to
estimating $\int Q(P-\Phi_T*P)\,d\mu$.
Write $\xi'$ and $\xi$ for frequencies of $Q$ and $P$.
For $\|\xi\|\le T/\Lambda$, the multiplier
$1-\widehat\Phi(\xi/T)$ is $O(\Lambda^{-2})$.
For $\|\xi\|>T/\Lambda$, we have
$\|\xi+\xi'\|\ge\|\xi\|/2\ge T/(2\Lambda)$ for large $N$, because
$\|\xi'\|\le N^{C_A}a_*$ and $T\ge a_*\Lambda^2$.
Fourier decay bounds the corresponding integral by $O(\Lambda^{-\eta})$.
The product of the coefficient norms is at most $e^{C_A\ell^5}$;
therefore the total is at most
$Ce^{C_A\ell^5}(\Lambda^{-2}+\Lambda^{-\eta})=O_A(N^{-A})$.

For the insertion estimate, take the polynomial $\psi_E$ from
\cref{lem:smoothing} with $B$ sufficiently large and mean $v+O(N^{-B})$.
Since $|D|\ge\Lambda^2$, its approximation at frequency $D$
has arbitrary inverse-power $L^2$ accuracy.  After replacing the
weight by $Q$ and the indicator by $\psi_E$, a nonconstant Fourier mode
has frequency $Dk+\xi'$, where $k\in\Z^d\setminus\{0\}$ and
$\|\xi'\|\le N^{C_A}a_*$.  Its norm is at least $|D|/2$.
Fourier decay and the coefficient bounds again make the total
nonconstant contribution $O_A(N^{-A})$.  The constant term gives
$v\int F\,d\mu$ to the same accuracy.  The integrated versions use
Cauchy--Schwarz on the product of $\mu$ and the parameter measure.
\end{proof}

For inserted hits with $|D|\ge\Lambda$, \cref{lem:smoothing} gives the
required approximations uniformly in the target. For smaller
positive frequencies $D\ge a_1$, we average the errors over target parameters.
Let $(\Theta,\nu)$ be either $\T^d$ with Haar measure, parametrizing
fixed-body translates, or the slab parameter space in
\eqref{eq:model-slabs} with normalized Lebesgue measure. Write
$J_\vartheta$ for the corresponding target. The bounded smoothing
polynomials $\psi_\vartheta$ at width
$\delta$ satisfy, for $D\ge a_1$,
\begin{equation}\label{eq:early-hit-smoothing}
 \int_\Theta\|\psi_\vartheta(D\,\cdot)-\1_{J_\vartheta}(D\,\cdot)\|_{L^2(\mu)}^2
       \,d\nu(\vartheta)
 \ll \delta+D^{-\eta}\delta^{\eta_0}
 \ll_{a_1}\delta^{\eta_0},\qquad
 \eta_0=\tfrac12\min\{\eta,1\}.
\end{equation}
Boundedness reduces squared error to $L^1$ error.  For translates,
integration in the parameter gives $O(\delta)$ pointwise in $x$, by
\eqref{eq:smooth-target}. For slabs, integration over the intercept
also gives $O(\delta)$ unless a coordinate of $Dx\bmod1$ lies within
$O(\delta)$ of a vertical cube face.  Each such coordinate layer has
$\mu$-mass $O(\delta+D^{-\eta}\delta^{\eta_0})$ by \cref{lem:interval}.
Taking $\delta=N^{-B}$ with $B$ sufficiently large verifies
the parameter-integrated hypothesis in \cref{ac:slowweights,ac:transfer},
including for initially omitted indices.

\section{Local hitting estimates and convex geometry}\label{sec:all-convex-geometry}
We estimate hitting probabilities using local averages.
Write $X_n(x)=a_nx\bmod1$ and
$\mathbb P_Q(E)=\vol(E\cap Q)/\vol(Q)$ for an axis-parallel cube $Q$;
let $\mathbb E_Q$ denote the corresponding expectation.
Constants are uniform over all normalized convex targets; for a fixed
body they may also depend on $\mathcal C$.

\subsection{Finite inner approximations}

We use the inner approximating family from
\cref{ac:normalization,ac:catalogue}, together with its anisotropic scales
and frames.  Stefanescu~\cite[Step~5]{Stefanescu2025MZ} discretizes
inscribed hyperrectangles by scale, rotation, and translation.
For homothets of a fixed body, the following smaller family suffices.

\begin{lemma}\label{lem:fixed-catalogue}
For each fixed $0<\varepsilon<1$ and sufficiently small $s$, there is
a collection of $O_{\mathcal C,\varepsilon}(s^{-1})$ translates of
$\mathcal C_{(1-\varepsilon)s}$ such that every translate of
$\mathcal C_s$ in $\T^d$ contains one of them.
Moreover, if a point $p$ lies in the fixed enlargement of a translate's
enclosing box, the translate contains $p+J$ for
some $J$ in a list of
$O_{\mathcal C,\varepsilon}(1)$ homothets of $\mathcal C$,
independent of $p$, all of volume $(1-\varepsilon)^2s$.
The relative family can be labeled so that, for each $p$, every incident
original translate is assigned to a contained $p+J$, with at most
$C_{\mathcal C,\varepsilon}$ original translates per label.
\end{lemma}
\begin{proof}
Choose $b>0$ with $bB_2\subset\mathcal C$, and put
$u=(1-\varepsilon)^{1/d}$.  A center error of size at most
$b(1-u)s^{1/d}$ is absorbed by the contraction $u\mathcal C_s$.
A torus mesh of this spacing has $O(s^{-1})$ points.
If an enlarged enclosing box contains $p$, the displacement of
the center from $p$, divided by $s^{1/d}$, lies in a fixed bounded
set.  Mesh this set at a fixed small spacing and apply a second contraction
by the factor $u$.  The inward margin absorbs the displacement error,
and the resulting relative list has bounded cardinality and volume
$(1-\varepsilon)^2s$.  Keep the displacement-mesh point as part of each
label.  For fixed $p$ and a fixed label, the original center lies in a
box of diameter $O(s^{1/d})$; the original center mesh has spacing
comparable to $s^{1/d}$.  It therefore contributes only a bounded number
of centers, proving the multiplicity assertion.
\end{proof}

\subsection{Pair intersections}
The following estimates are stated for targets in $[0,1]^d$, but the same
proofs apply to convex lifts of coordinate widths at most one projected
modulo $\Z^d$.  The projection is injective on the interior, and unit-cell
intersections give $O_d(1)$ convex pieces.  All cycle lists and adapted-box
constants are uniform over the finite normalized families used below.
Let $Q$ be an axis-parallel cube of side $\rho$.

\begin{lemma}\label{ac:rotation-return}
For every $0<\varepsilon<1$ there exist an integer $H\ge1$ and
constants $C,T_0>0$ with the following property. Let $r\in\R$ and let $I\subset\Z$ be an interval of
$T\ge T_0$ consecutive integers. If, for some arc $J\subset\T$ of
length $0<w\le\varepsilon/8$,
\[
 \#\{k\in I:rk\bmod1\in J\}\ge\varepsilon T,
\]
then some $1\le m\le H$ satisfies
\begin{equation}\label{ac:rotation-capture}
 T\|mr\|_{\R/\Z}\le Cw.
\end{equation}
\end{lemma}
\begin{proof}
This is \eqref{od:rotation-count} in the proof of
\cref{lem:od-pairs}, with $\epsilon'=\varepsilon$, $P=T$, and $v=w$.
\end{proof}

\begin{lemma}\label{ac:pairs}
For convex $K,K'\subset[0,1]^d$ of volume $s$ and $a<b$ with $a\rho\ge4$,
\begin{equation}\label{ac:coarse}
 \mathbb P_Q(ax\bmod1\in K,bx\bmod1\in K')
 \le C_d(s^2+s a/b).
\end{equation}
For $K'=K$, fix $\varepsilon>0$ and an upper bound $R$ for $b/a$.
There exist an integer $H\ge1$ and constants $C,T_0,s_0>0$,
depending only on $d,\varepsilon,R$,
such that, if $0<s<s_0$ and $a\rho\ge T_0$, a pair probability
exceeding $\varepsilon s$ implies the existence of coprime positive
integers $m,n\le H$ with
\begin{equation}\label{ac:capture}
 |mb-na|\rho\le Cw,
 \qquad w=\ell_d,
\end{equation}
where $\ell_d$ is the shortest scale of $K$ in
\eqref{ac:normalize}.
In particular, the finite list of candidate integer pairs is independent
of the target's center, orientation, and normalized body.
\end{lemma}
\begin{proof}
Inside a unit cell for $ax$, the fast-cell mesh is $a/b$.
The number of its cells meeting the slow target body is at most
$C_d(b/a)^d(s+a/b)$; each contributes at most $s(a/b)^d$.
The outer-layer bound in \eqref{eq:uniform-boundary} gives
$\vol(K+[-t,t]^d)\le s+C_dt$. Covering $aQ$ by full unit cells
changes the resulting bound by at most a dimension-dependent factor.
This proves \eqref{ac:coarse}.

For the refined estimate, put $r=b/a\in[1,R]$ and $T=a\rho$.
Full slow cells have $ax=k+z$, where $z\in K$ and the coordinates
of $k$ range independently over integer intervals of length $T+O(1)$.
The partial boundary cells contribute $O_d(s/T)$.
A lift of $K$ lies in a slab
\[
 \{y:|\langle\nu,y\rangle-c|\le C_dw\},
 \qquad \|\nu\|=1,\qquad w=\ell_d.
\]
Choose a coordinate $j$ with $|\nu_j|\ge d^{-1/2}$.
For fixed $z$ and all the other coordinates of $k$, the second-hit
condition $r(k+z)\bmod1\in K$ restricts $rk_j\bmod1$ to an arc
of length at most $C_dw$, by the coordinate-fiber bound for the slab.
If the pair probability exceeds $\varepsilon s$, then, after discarding
the boundary error and taking $T$ sufficiently large, some such fiber
contains at least a fixed positive fraction, depending only on
$d,\varepsilon$, of this one-dimensional rotation orbit.
For sufficiently small $s$, one has $w\ll_d s^{1/d}$ and this arc is
small enough for \cref{ac:rotation-return}. Hence there is a bounded
positive integer $m$ such that
\[
 T\|mr\|_{\R/\Z}\le Cw.
\]
Let $n$ be the nearest integer to $mr$. Since $1\le r\le R$, both
$m$ and $n$ are positive and bounded in terms of $d,R,\varepsilon$.
Dividing by their greatest common divisor, if needed, gives coprime
$m,n$ with
$|mb-na|\rho=T|mr-n|\le Cw$, as required.
For a convex lift of coordinate widths at most one, the projected fibers
are unions of $O_d(1)$ arcs; one arc then contains a fixed positive
fraction of the visits.  The same conclusion follows after changing
only the constants.
\end{proof}

\subsection{Generic and uniform hitting estimates}
For a normalized target $K=c_K+U\mathsf D\mathcal P$ of volume $s$, put
\[
 R(K)=c_K+U\prod_{j=1}^d[-C_d\ell_j,C_d\ell_j],
 \qquad R_1(K)=c_K+A_1(R(K)-c_K),
\]
where the fixed enlargement $A_1$ will be chosen below.
Thus $\vol(R_1(K))=O(s)$, and its shortest principal scale is comparable
to $w=\ell_d$.

\begin{lemma}\label{ac:generic-pair-hit}
Fix $d,q$, a positive integer $M$, and $\varepsilon_1>0$.
There exist an integer $H_1\ge1$ and constants $\varepsilon,A_1,T_0>0$
with the following property
for all sufficiently small $s$. Let $\mathcal I$ consist of $M$ indices
of the lacunary sequence, let $a_{\min}$ be its smallest frequency,
and suppose $Q$ has side $\rho$
with $a_{\min}\rho\ge T_0$.  Call $\mathcal I$ exceptional for $K,Q$
if some pair of its hit events has intersection probability exceeding
$\varepsilon s$.  If it is not exceptional, then
\[
 \mathbb P_Q(\exists n\in\mathcal I:X_n\in K)
       \ge(1-\varepsilon_1)Ms.
\]
If it is exceptional, there are two of its frequencies $a<b$ and bounded
positive integers $m,n,D\le H_1$, with $m\ne n$, such that, after possibly
changing the sign of $z=mb-na$,
\begin{equation}\label{ac:pair-incidence}
 |z|\rho\le Cw,\qquad
 zx\bmod1\in D R_1(K)\pmod{\mathbb Z^d}\quad\hbox{for every }x\in Q.
\end{equation}
The lists of pairs and integer coefficients are independent of $K$.
\end{lemma}
\begin{proof}
Choose $\varepsilon\ll\varepsilon_1/M$.
Taking $T_0$ sufficiently large, period counting gives every marginal
probability at least $(1-\varepsilon_1/2)s$.  If all pair probabilities are at
most $\varepsilon s$, Bonferroni proves the first assertion.
Choose a ratio cutoff so large that \eqref{ac:coarse} gives this same
small-pair bound above the cutoff, for all sufficiently small $s$.
For an exceptional pair below it, \cref{ac:pairs} supplies bounded
coprime $m,n$ and $z=mb-na$ with $|z|\rho\le Cw$.
Lacunarity excludes $m=n$: in that case $m=n=1$ and
$|z|\rho\ge(q-1)a_{\min}\rho$, contradicting the preceding bound.

An exceptional pair has a simultaneous hit at some $x_0\in Q$.
Writing those two hits as $y,y'\in K$ gives
\[
 zx_0\bmod1=(my'-ny)\bmod1
 \in (m-n)c_K+U\prod_{j=1}^d[-C'\ell_j,C'\ell_j]
                       \pmod{\mathbb Z^d}.
\]
Change the sign if necessary and put $D=|m-n|$.
For $x\in Q$, the additional displacement $z(x-x_0)$ has norm $O(w)$,
which is absorbed by a fixed enlargement of every principal scale.
Choosing $A_1,H_1$ large enough proves \eqref{ac:pair-incidence}.
\end{proof}

\begin{lemma}\label{ac:positive-hit}
There is a constant $\chi=\chi(d,q)>0$ such that, whenever
$Ls\le1$ and $a_{\min}\rho$ is sufficiently large,
\[
 \mathbb P_Q(\exists n\in\mathcal I:X_n\in K)\ge\chi Ls
\]
for every $L$-index block $\mathcal I$ and every normalized convex target
$K$ of volume $s$.
\end{lemma}
\begin{proof}
For $Z=\sum_{n\in\mathcal I}\1_K(X_n)$, period counting gives
$\mathbb E_QZ\ge\tfrac12Ls$.  By \eqref{ac:coarse} and lacunarity,
\[
 \mathbb E_QZ^2\le C_{d,q}(Ls+L^2s^2)\le 2C_{d,q}Ls,
\]
since $\sum_{i<j}a_i/a_j\le C_qL$.
Cauchy--Schwarz, in the form
$\mathbb P_Q(Z>0)\ge(\mathbb E_QZ)^2/\mathbb E_QZ^2$, proves the claim.
\end{proof}

\subsection{Rational-rank entropy and pinned approximations}
The rational-rank stratification gives sharper hitting coefficients and
smaller approximating families for the all-convex upper bound.
\begin{lemma}\label{ac:rankentropy}
Fix $d\ge1$, a finite set $\mathcal V\subset\R^d\setminus\{0\}$,
and constants $A,L_0>0$. Consider the finite inner approximating family obtained
from the normalization $K=c_K+U\mathsf D\mathcal P$, with $U\in SO(d)$,
$\mathsf D=\operatorname{diag}(\ell_1,\ldots,\ell_d)$,
$L_0\ge\ell_1\ge\cdots\ge\ell_d>0$, and a uniformly bounded number of
normalized bodies for each choice of scales.  Suppose
$\det \mathsf D\asymp s$, and construct its
orientation nets as disjoint packings in the neighborhoods
$\{R:\|\mathsf D^{-1}(R-I)\mathsf D\|\le\delta\}$.
Associate each member with its chosen $U,\mathsf D,c_K$, including after
the final homothetic adjustment to its common volume. Define its rank by
\begin{equation}\label{eq:rational-rank}
 \operatorname{rk}(K)=\dim\operatorname{span}\{v\in \mathcal V:
             |(U\mathbf e_j)\cdot v|\le A\ell_j\ (1\le j\le d)\}.
\end{equation}
For sufficiently small $s$, one has $\operatorname{rk}(K)\le d-1$. For every
$0\le h\le d-1$, the subcollection with $\operatorname{rk}(K)\ge h$ has cardinality
$$
 O\bigl(s^{-(d-h)}(\log(1/s))^{C_d}\bigr).
$$
If a fixed enlargement of the enclosing box is additionally required
to contain a prescribed point $p$, the corresponding cardinality is
$$
 O\bigl(s^{-(d-h-1)}(\log(1/s))^{C_d}\bigr).
$$
The implied constants may depend on $d,\mathcal V,A,L_0$ and the approximation accuracy,
but not on $s$. Both estimates count subfamilies of the family fixed above.
\end{lemma}
\begin{proof}
Fix the scales and an independent $h$-tuple $v_1,\ldots,v_h$ from
$\mathcal V$; there are finitely many choices.  We claim
\begin{equation}\label{ac:rankhaarmeasure}
 \operatorname{Haar}\{U: |(U\mathbf e_j)\cdot v_i|\le A\ell_j
                         \text{ for all }i,j\}
       \ll(\det\mathsf D)^h.
\end{equation}
For $h=0$ this is immediate.  Otherwise let
$V=[v_1\ \cdots\ v_h]$.  Since $\mathcal V$ is finite and the chosen tuple is
independent, write $V=WB$, where $W$ is an orthonormal $h$-frame and
$B$ is invertible with $\|B^{-1}\|\ll_{\mathcal V,h}1$.  The constraints then give
$\|\operatorname{row}_j(U^{-1}W)\|\le C\ell_j$.

For Haar-uniform $U\in SO(d)$, $U^{-1}W$ is Haar-uniform on the
Stiefel manifold. By Gaussian polar factorization~\cite[Proposition~9]{Tropp2012},
a standard Gaussian $d\times h$ matrix $G$ admits $G=QR$, where
$Q$ has the Haar-Stiefel law and is independent of
$R=(G^{\mathsf T}G)^{1/2}$.

Let $E_Q$ denote the above row-bound event for
$Q$, and define $E_G$ similarly.  On $\{\|R\|\le1\}$, which has fixed
positive probability and is independent of $Q$, the event $E_Q$ implies
$E_G$.  Hence independence of the Gaussian rows gives
\[
 \mathbb P(E_Q)\ll\mathbb P(E_G)
 =\prod_{j=1}^d\mathbb P(\|Z\|\le C\ell_j)
 \ll\prod_{j=1}^d\ell_j^h,
 \qquad Z\sim N(0,I_h).
\]
This proves \eqref{ac:rankhaarmeasure}.

The disjoint orientation neighborhoods from \cref{ac:catalogue}
have Haar measure $\gg\prod_{j<k}\ell_k/\ell_j$.
A neighborhood centered at a qualifying $U$ satisfies the same
constraints with enlarged $A$, since for an orientation perturbation $S$,
$$
 \mathsf D^{-1}(US)^{-1}v_i
 =(\mathsf D^{-1}S^{-1}\mathsf D)(\mathsf D^{-1}U^{-1}v_i).
$$
Dividing \eqref{ac:rankhaarmeasure} by this neighborhood measure
bounds the number of qualifying orientations by
$$
 C(\det\mathsf D)^h\prod_{j<k}\frac{\ell_j}{\ell_k}
 \le C'(\det\mathsf D)^{-(d-h-1)}\asymp s^{-(d-h-1)}.
$$
The inequality follows after multiplying by
$(\det\mathsf D)^{d-h-1}$: the product is
$\prod_j\ell_j^{2(d-j)}$, bounded by $\ell_j\le L_0$.
The adapted center mesh contributes $O(s^{-1})$ points.  If a fixed
enlargement of the adapted box is required to contain a prescribed point
$p$, then $\mathsf D^{-1}U^{-1}(c_K-p)$ ranges in a fixed bounded box.
Since the center mesh has spacings comparable to the $\ell_j$, only
$O(1)$ mesh points are then possible for each frame and scale.  Scale
choices contribute logarithms and normalized bodies contribute a fixed
factor.  These are the two asserted counts.

Finally, if $d$ independent vectors from $\mathcal V$ satisfy the constraints,
Hadamard's inequality gives
\[
 0<c_{\mathcal V}\le|\det[v_1\ \cdots\ v_d]|
       \le d^{d/2}A^d\prod_j\ell_j\ll s,
\]
a contradiction for small $s$.  Thus the rank is at most $d-1$.
\end{proof}

\begin{lemma}\label{ac:relative-pinned}
Fix $\mathcal V,A,L_0$ and an approximating family as in \cref{ac:rankentropy}.
Define all ranks using these $\mathcal V,A$ and each member's associated
$U,\mathsf D$.  For each fixed
$0<\varepsilon<1$ and $1\le k\le d$, there is a collection
$\mathscr J_{s,k}$ of at most
$$
 C s^{-(k-1)}(\log(1/s))^{C_d}
$$
convex bodies with associated scales and frames of rank $d-k$, independent
of $p\in\R^d$, with the following property.  If $K$ is a member of this
fixed family, has volume $s$ and rank $d-k$, and a fixed enlargement of
its adapted box contains $p$, then $K$ contains $p+J$ for some
$J\in\mathscr J_{s,k}$ of volume $(1-\varepsilon)s$.  In this containment,
$J$ retains $K$'s $U,\mathsf D$ for subsequent rank comparisons.
For $K\subset[0,1]^d$, the relative bodies may be restricted to have
coordinate widths at most one.  A common fixed enlargement of the
adapted boxes applies to both families.
Moreover, the relative bodies can be labeled so that for every $p$ the
incident original bodies admit an assignment $K\mapsto J_p(K)$ satisfying
$ p+J_p(K)\subset K$ and having uniformly bounded fibers.  In particular,
for any finite point set $\mathcal X$,
\begin{equation}\label{ac:pinned-multiplicity}
 \sum_{\substack{K:\,p\in R_1(K)\\\operatorname{rk}(K)=d-k}}
       \1_{\{\operatorname{int}K\cap\mathcal X=\varnothing\}}
 \le C\sum_{J\in\mathscr J_{s,k}}
       \1_{\{\operatorname{int}(p+J)\cap\mathcal X=\varnothing\}}.
\end{equation}
Here $R_1(K)$ denotes any fixed enlargement of the adapted box, and the
sum on the right is over labels, allowing repeated underlying bodies.
\end{lemma}
\begin{proof}
Keep the scale and orientation meshes from \cref{ac:rankentropy}.
In the coordinates $U\mathsf D$, the displacement $c_K-p$ lies in
a fixed bounded box.  Mesh this displacement at a fixed small
spacing and use the finite normalized inner approximations from
\cref{ac:normalization}.  A reserved inward margin absorbs the
displacement error.  This requires only a fixed number of offsets
and normalized bodies at each frame and scale.  The orientation count in
\cref{ac:rankentropy}, with
$h=d-k$, is $O(s^{-(k-1)})$, and the number of scale choices is polynomial
in $\log(1/s)$.
Choose the meshes sufficiently fine and contract to the specified
volume.  Retain the original $U,\mathsf D$ throughout; only the offset
and normalized inner body change.  Since \eqref{eq:rational-rank} depends
only on $\mathcal V,A,U,\mathsf D$, the relative body has the same rank as $K$.
Retain in each relative label its scale, frame, original normalized
body and displacement-mesh point.  If $p$ and this label are fixed, the
original center $c_K$ lies in a single adapted displacement cell.
The global center mesh has spacings comparable to $\ell_j$ in the same
frame, so this cell contains only $O(1)$ original centers.  There are only
a bounded number of normalized bodies at each scale and frame; retaining
their labels therefore proves the bounded-fiber assertion and
\eqref{ac:pinned-multiplicity}.
Uniform boundedness of the normalized bodies gives a common enlargement
factor for the adapted boxes.  For $K\subset[0,1]^d$, containment of
$p+J$ in $K$ ensures that $J$ has coordinate widths at most one, so
discarding wider bodies preserves every assignment.
\end{proof}

\subsection{Rational cycles and affine hitting}
For positive integers $m,k$, put $H_m^{(k)}=(m^{-1}\Z/\Z)^k$.

\begin{lemma}\label{ac:grids}
Let $k\ge2$. For distinct positive integers $m_1,\ldots,m_r$ and rational points
$t_1,\ldots,t_r\in(\Q/\Z)^k$, assign weight $m_j^{-k}$ to each point of
$t_j+H_{m_j}^{(k)}$, taking the maximum of the assigned weights at
coincidences. The total weight is at least
\begin{equation}\label{ac:gridbound}
 \sum_{j=1}^r\frac{J_k(m_j)}{m_j^k}
 \ge \frac r{\zeta(k)},
 \qquad J_k(m)=m^k\prod_{p\mid m}(1-p^{-k}),
\end{equation}
where $J_k$ is the Jordan totient function and the product is over the prime
divisors of $m$.
\end{lemma}
\begin{proof}
For finite sets $C_j$ with weights $w_j>0$, the weighted union satisfies
\begin{equation}\label{ac:layercake}
 \sum_x\max_j w_j\1_{C_j}(x)
 =\int_0^\infty \#\bigcup_{j:w_j\ge u}C_j\,du.
\end{equation}
We first show that, for every subfamily, translating the torsion grids cannot
make their union smaller than in the aligned case.

Choose $D$ divisible by the
$m_j$ and by the denominators of the shifts, and regard the grids as cosets
in $H_D^{(k)}$. Under the primary decomposition
$$
 H_D^{(k)}=\prod_{p\mid D}G_p,
$$
the $p$-primary part of $H_{m_j}^{(k)}$ is
$$
 G_{p,j}=\{g\in G_p:p^{v_p(m_j)}g=0\},
$$
where $v_p(m_j)$ is the exponent of $p$ in $m_j$. These subgroups
have cardinality $p^{k v_p(m_j)}$ and, for fixed $p$, are linearly
ordered by inclusion. Fix all
primary coordinates except the $p$-coordinate. If $I$ is the set of labels
whose other coordinates contain the chosen point, the corresponding fiber
of the union is a union of cosets of the nested groups $G_{p,j}$,
$j\in I$, and therefore has cardinality at least
$$
 \max_{j\in I}\#G_{p,j}.
$$
Replacing every $p$-primary coset by the subgroup $G_{p,j}$ itself makes
that fiber exactly the largest active subgroup, so it cannot increase the
union. Iterating over the primes dividing $D$ shows that
$$
 \#\bigcup_j(t_j+H_{m_j}^{(k)})
 \ge \#\bigcup_j H_{m_j}^{(k)}
$$
for every rationally shifted subfamily.  Apply this inequality to the
active labels at each threshold in \eqref{ac:layercake}.
In the aligned family, the points of exact denominator $m_j$ are disjoint
for distinct $j$.  Their number is $J_k(m_j)$, and each has weight at
least $m_j^{-k}$, proving the first inequality in \eqref{ac:gridbound}.
Since $k\ge2$, the Euler product gives
$$
 \frac{J_k(m)}{m^k}=\prod_{p\mid m}(1-p^{-k})
 \ge \prod_p(1-p^{-k})=\frac1{\zeta(k)},
$$
which proves the second inequality.
\end{proof}

\begin{lemma}\label{ac:affine}
Fix positive integers $H,r$, distinct integers $1\le m_1,\ldots,m_r\le H$,
and arbitrary $t_1,\ldots,t_r\in\T^d$, represented in $[0,1)^d$.
For Haar-uniform $y\in\T^d$, consider the maps $y\mapsto m_jy+t_j$.
There exist a finite set $\mathcal V\subset\Q^d\setminus\{0\}$ and a constant
$A>0$, depending only on $d,H,r$, such that the following holds. Let
$K=c_K+U\mathsf D\mathcal P$ be a convex lift of coordinate widths at most one,
with bounded scales and normalized bodies as in
\cref{ac:normalization,ac:relative-pinned}. Put $s=\vol(K)$ and let $h=\operatorname{rk}(K)$
be the rank in \eqref{eq:rational-rank} with these $\mathcal V,A$.
Membership in $K$ below is understood modulo $\Z^d$.
Suppose $s$ is sufficiently small.
Then $h\le d-1$, and, if $k=d-h\ge2$,
$$
 \mathbb P(\exists j:m_jy+t_j\in K)\ge \frac{rs}{\zeta(k)}.
$$
If the inverse branches have pairwise disjoint interiors, the probability
is exactly $rs$.
\end{lemma}

\begin{proof}
After translating $K$ by an integer vector, assume $c_K\in[0,1)^d$.
The inverse branches are the torus images of
$$
 E_{j,b}=(K-t_j+b)/m_j,\qquad b\in\{0,\ldots,m_j-1\}^d.
$$
Branches of the same label have disjoint interiors and volume $s/m_j^d$.
Join branches having an interior intersection. For a component $C$,
let $N_j(C)$ count its branches of label $j$. Thus
\begin{equation}\label{ac:componentmass}
 \mathbb P(\text{some hit})\ge
 s\sum_C\max_j\frac{N_j(C)}{m_j^d}.
\end{equation}
Write $s_j=(c_K-t_j)/m_j$. Every overlap produces an edge on the label
set, retaining parallel edges, with displacement
$$
 d_e=s_j-s_k+\delta_e,\qquad |(U\mathbf e_i)\cdot d_e|\le C_d\ell_i.
$$
Here $\delta_e=b/m_j-b'/m_k+z$ is a rational vector from a fixed finite
list.  With the above choice of lift, all branch centers lie in a fixed
bounded set and the branch diameters are uniformly bounded, so only finitely
many integer torus lifts $z$ can occur.  Choose a spanning forest on the label
multigraph and rational potentials
$p_j$ recursively so that $p_j-p_k=\delta_e$ on every forest edge
$e=(j,k)$.  For any edge $e=(j,k)$ put
$$
 v_e=\delta_e-p_j+p_k.
$$
Thus $v_e=0$ on the forest.  For a nonforest edge, subtracting the overlap
relations along the unique forest path from $j$ to $k$ cancels the terms
$s_j-s_k$ and gives
\begin{equation}\label{ac:cycle}
 |(U\mathbf e_i)\cdot v_e|\le C_{d,H,r}\ell_i.
\end{equation}
All possible nonzero $v_e$, over the finitely many label multigraphs and
edge vectors, form a fixed finite rational set $\mathcal V$. Choose $A$ large enough
that every inequality in \eqref{ac:cycle} is included in the rank test
\eqref{eq:rational-rank}.

Let $V$ be the span of the cycle vectors and put $h'=\dim V$.  By
\eqref{ac:cycle}, every cycle vector is counted in the rank defining $h$,
so $h'\le h$.  By \cref{ac:rankentropy}, $h\le d-1$, hence $h'<d$.
For counting purposes, assign to each branch $E_{j,b}$ the point
$b/m_j-p_j$ modulo $\Z^d$. This label-dependent recentering need not
preserve intersections. If two original branches overlap, their assigned
points differ by $\delta_e-p_j+p_k=v_e$ modulo $\Z^d$. Hence they
have the same image in $\R^d/(V+\Z^d)$, and every original overlap
component has a single quotient image. As $V$ is rational, this quotient
is a torus of dimension $k'=d-h'$. Merge components with the same
quotient image. This can only decrease the sum in
\eqref{ac:componentmass}: if $\mathcal D$ is such a class of components
and $b_{Cj}=N_j(C)/m_j^d$, then
$$
 \max_j\sum_{C\in\mathcal D}b_{Cj}
 \le \sum_{C\in\mathcal D}\max_j b_{Cj}.
$$
$V\cap\Z^d$ is a primitive rank-$h'$
sublattice of $\Z^d$.  Extend a basis of it to an integer basis of
$\Z^d$.  In these coordinates the quotient map sends
$H_{m_j}^{(d)}$ onto $H_{m_j}^{(k')}$, and every quotient point has exactly
$m_j^{h'}$ preimages.  Hence, for each quotient component, the total number
of label-$j$ branches among all old components merged into it is exactly
$m_j^{h'}$ when that quotient point belongs to the projected label-$j$ grid,
and is zero otherwise.  After multiplying by the original branch weight
$m_j^{-d}$, the resulting weight is $m_j^{-k'}$.  The projected shifts
are rational because the $p_j$ are rational and the quotient coordinates
come from an integer basis.  Thus, by
\cref{ac:grids}, \eqref{ac:componentmass} is at least
$rs/\zeta(k')$ when $k'\ge2$.  Since $h'\le h$, one has $k'\ge k$, and
$1/\zeta(k')\ge1/\zeta(k)$ for $k'\ge k\ge2$.  The assertion follows.
The disjoint-branch assertion follows by adding their volumes.
\end{proof}

\subsection{Rational components and block hitting}
We group frequencies linked by bounded rational relations.  Graphs
encoding rational relations also appear in
Alon--Peres~\cite[Section~5]{AlonPeres1992}.

\begin{lemma}\label{ac:rational-cluster}
Fix positive integers $M,H$ and a constant $C>0$. There exist an
integer $H'\ge1$ and a constant $C'>0$, depending only on $M,H,C$,
with the following property. Let
$G$ be a connected graph on at most $M$ labels, let $a_j>0$ be attached
to its vertices, and let $\rho,w>0$. Suppose that for every edge
$ij$ of $G$ there are integers $p_{ij},q_{ij}\in\{1,\ldots,H\}$ such that
\begin{equation}\label{ac:cluster-relation}
 |p_{ij}a_j-q_{ij}a_i|\rho\le Cw.
\end{equation}
Then there are $v>0$, integers $1\le m_j\le H'$, and real numbers $e_j$
such that
\begin{equation}\label{ac:cluster-representation}
 a_j=m_jv+e_j,\qquad |e_j|\rho\le C'w.
\end{equation}
If, in addition,
\begin{equation}\label{ac:cluster-separation}
 |a_i-a_j|\rho>2C'w\qquad(i\ne j),
\end{equation}
then the integers $m_j$ are pairwise distinct.
\end{lemma}
\begin{proof}
Choose a root vertex $j_0$ and a spanning tree. Propagating
\eqref{ac:cluster-relation} along its paths gives
\[
 a_j=\alpha_j a_{j_0}+e_j,\qquad |e_j|\rho\le C_0w,
\]
where each positive rational $\alpha_j$ has numerator and denominator
at most $H^{M-1}$.  Along each edge the accumulated error is
multiplied by at most $H$ and increased by at most $Cw/\rho$; there are
at most $M-1$ edges in a path.
Let $D$ be the product of the reduced denominators of the $\alpha_j$.
Then $D$ is bounded in terms of $M,H$, and
$v=a_{j_0}/D$, $m_j=D\alpha_j$ give \eqref{ac:cluster-representation}, with
bounded positive integers $m_j$.
If $m_i=m_j$, then $|a_i-a_j|\rho=|e_i-e_j|\rho\le2C_0w$,
contrary to \eqref{ac:cluster-separation} after taking $C'\ge C_0$.
\end{proof}

\begin{lemma}\label{ac:packet}
Fix $d,q$, a positive integer $M$, and $\varepsilon_1>0$. There exist
a finite set $\mathcal V\subset\Q^d\setminus\{0\}$ and constants $A,T_0>0$,
depending only on $d,q,M,\varepsilon_1$, such that the following holds
for all sufficiently small $s$.  Use the frames from \cref{ac:catalogue}, define $\operatorname{rk}(K)$ by
\eqref{eq:rational-rank} with this fixed $\mathcal V,A$, and put $k(K)=d-\operatorname{rk}(K)$.
For any increasing $M$-tuple of positive frequencies with successive
ratios at least $q$, let $a_{\min}$ denote the smallest of them.
On every averaging cube $Q$ of side $\rho$
satisfying
$$
 \rho a_{\min}\ge T_0,
$$
a catalogue target $K$ of volume $s$ and $k(K)\ge2$ has hitting probability at least
$$
 (\chi_{k(K)}-\varepsilon_1)Ms,\qquad \chi_k=1/\zeta(k)\quad(k\ge2).
$$
\end{lemma}
\begin{proof}
Choose $\varepsilon\ll\varepsilon_1/M$ and a ratio cutoff so large that
\eqref{ac:coarse} bounds every pair above it by $\varepsilon s$.
For the remaining ratios take the bounded coefficients $H$ and error
constant $C$ from \cref{ac:pairs}, and let $H',C'$ be the constants in
\cref{ac:rational-cluster}.  Apply \cref{ac:affine} for each
$r\le M$ and multipliers at most $H'$; take the union of its finite sets
as $\mathcal V$.  Choose $A$ large enough for all these applications and the
contractions below.  Fix $\mathcal V,A$ before
dividing targets into rank classes; choose $T_0$ sufficiently large below.

Join two labels whenever they have a bounded-coefficient relation
\eqref{ac:capture}.  Pairs in different components have joint-hit
probability at most $\varepsilon s$, by the ratio cutoff and
\cref{ac:pairs}.  On each component, \cref{ac:rational-cluster} gives
\begin{equation}\label{ac:representation}
 a_j=m_jv+e_j,\qquad 1\le m_j\le H',\qquad |e_j|\rho\le C'w.
\end{equation}
The $m_j$ are distinct: equality would give
$|a_i-a_j|\rho\le2C'w$, whereas lacunarity gives
$|a_i-a_j|\rho\ge(q-1)T_0$ and $w\to0$.

Choose $\varepsilon'>0$ sufficiently small in terms of $\varepsilon_1$,
and subdivide $Q$ into cubes of side $\rho/B$, where the fixed integer
$B$ is large enough that each $e_jx$ varies by at most $\varepsilon'w$
on a small cube.  For such a cube $Q'$ with center $x_{Q'}$, put
$t_j=e_jx_{Q'}$ and $K^- =c_K+(1-\tau)(K-c_K)$, where
$\tau=O_d(\varepsilon')$ is chosen so that
$K^-+\varepsilon'w B_2\subset K$; this follows from
$c_K+w B_2\subset K$.  Thus for $x\in Q'$,
$ m_jvx+t_j\in K^- $ implies
$ a_jx=m_jvx+e_jx\in K $, with all memberships taken modulo
$\mathbb Z^d$.  The volume of $K^-$ is $(1-O_d(\varepsilon'))s$.
On each small cube, period counting shows that $vx\bmod1$ has density
$1+O(BH'/T_0)$ relative to Haar measure, since
$v\rho\ge(a_{\min}\rho-C'w)/H'\ge T_0/(2H')$ for small $s$.
Choose $T_0$ after $B$ to make this relative error sufficiently small.

For a component with $r_C$ labels, the frozen maps are
$y\mapsto m_jy+t_j$.  If $h_C$ is the rank used by \cref{ac:affine} for
the contracted target, then $h_C\le\operatorname{rk}(K)$: its finite
set is contained in $\mathcal V$ and its adapted-coordinate constant is covered
by the global choice of $A$.
Thus $k_C=d-h_C\ge k(K)$, and the affine estimate gives hit probability
at least
\[
 \frac{r_C(1-O_d(\varepsilon')-O(T_0^{-1}))s}{\zeta(k_C)}
 \ge \frac{r_C(1-O_d(\varepsilon')-O(T_0^{-1}))s}{\zeta(k(K))}.
\]
Average over the small cubes and apply the Bonferroni inequality to
the component hitting events on the original cube $Q$.  There are at most $M^2$ cross-component pairs,
whose total contribution is at most $M^2\varepsilon s$.
Our choices of $\varepsilon',\varepsilon$ and a sufficiently large $T_0$
absorb these errors in $\varepsilon_1Ms$.
\end{proof}

Fix a small $\varepsilon_1>0$ and a sufficiently large short-block length
$M$.  Let $L\to\infty$ and $Ls\to0$ as $s\to0$, and assume
$a_{\min}\rho\ge T_0$ for the chosen short-block estimates.  Partition
a block of $L$ indices into consecutive short blocks of length $M$,
dropping fewer than $M$ indices at the end.  At each boundary between successive short blocks, say after index $i_0$,
\[
 \sum_{i\le i_0<j}\frac{a_i}{a_j}
 \le\sum_{r\ge0,\,t\ge1}q^{-r-t}=C_q.
\]
Every pair from different short blocks crosses a boundary, so summing
this inequality over the $O(L/M)$ boundaries gives total ratio sum
$O_q(L/M)$.  Thus \eqref{ac:coarse} bounds the sum of cross-block
intersection probabilities by
\[
 C_dL^2s^2+C_{d,q}Ls/M.
\]
Apply \cref{ac:packet} with a sufficiently small loss.  Since
$L^2s^2=o(Ls)$ and $Ms=o(Ls)$, Bonferroni absorbs
the cross-block errors and the final incomplete short block as $s\to0$.
For every target in class $k\ge2$ this gives
\begin{equation}\label{ac:blockhit}
 \mathbb P_Q(\text{block hits }K)\ge(\chi_k-\varepsilon_1)Ls,
 \qquad\chi_k=1/\zeta(k).
\end{equation}
For $k=1$, and for a fixed body in any dimension, use
\cref{ac:positive-hit} directly, setting $\chi_1=\chi(d,q)$.
Thus a positive version of \eqref{ac:blockhit} holds for every target.

Let $B_b$ be the number of exceptional short blocks in the sense of
\cref{ac:generic-pair-hit}.  Apply that lemma to the other short blocks
and use the same cross-block estimate to obtain
\begin{equation}\label{ac:generichit}
 \mathbb P_Q(\text{block hits }K)
       \ge(1-\varepsilon_1)Ls-Ms B_b.
\end{equation}
Every exceptional short block contributes one of the bounded-coefficient
incidences \eqref{ac:pair-incidence}; the enlargement is fixed once and
for all.

To pass to ball averages, fix a small $\gamma>0$ and cover
at least a $1-\gamma$ fraction of a ball of radius $r$ by disjoint full
cubes of side $\delta r$, where $\delta>0$ depends only on $d,\gamma$.
Apply the rank-dependent lower bound on these cubes and discard the
uncovered part.  For the generic estimate, declare a short block
exceptional for the ball if it is exceptional on one of the cubes.
Then $|z|r\le Cw$, so \eqref{ac:pair-incidence} holds at the ball center
after a fixed enlargement.  Every other short block has the generic
lower bound on every full cube.  Choosing $\gamma$ small absorbs the
volume loss in both lower bounds.  The coarse pair estimate on balls
follows by comparison with an enclosing cube of bounded volume ratio;
together with period counting for the first moment it also proves
\cref{ac:positive-hit} for balls.  All lists and constants are independent
of the averaging radius.  For convolution with $\Phi_T$, use its
ball-mixture representation and discard radii below $C/a_{\min}$, with
$C\ge T_0/\delta$.  Their total mixture mass is
$O((T/a_{\min})^d)$, smaller than every inverse power of $N$ because
$a_{\min}/T\ge\Lambda$.  On each remaining full cube,
$a_{\min}\delta r\ge T_0$.

\section{Upper bounds}\label{sec:all-convex-upper}
For the all-convex upper bound, fix $c>d$ and a probability measure
$\mu$ satisfying \eqref{eq:fourier-decay}. Throughout this section,
$A_{\mathrm{err}}>0$
denotes an accuracy exponent, chosen sufficiently large after the geometric
parameters. We use $J=\operatorname{int}K$ for convex targets, so only
interior hits are counted.
Put $s=c\log N/N$ and $s_0=(1-\varepsilon)s$, the common volume of
the inner approximating family.  We use the short-block estimates with
$q_0=(q+1)/2$, so that the same constants also apply to the shifted
frequencies below.  Discarding finitely many initial indices, we may
suppose $a_n\ge2$.  This can only enlarge an avoidance event.

\subsection{Averaged hit rates and avoidance weights}
For a catalogue target $K$, use the enclosing boxes $R(K)$ and
$R_1(K)$ from Section~\ref{sec:all-convex-geometry}.
By \cref{ac:generic-pair-hit}, there exist an integer $H_1\ge1$
and a constant $A_1>0$, independent of $K$, such that every exceptional
short block on a cube $Q$ has a residual satisfying
\begin{equation}\label{ac:incidence}
 \begin{gathered}
 zx\bmod1\in E_D(K)\quad(x\in Q),\qquad z=\pm(ma_j-na_i),\\
 E_D(K)=D R_1(K)\pmod{\Z^d},\qquad 1\le m,n,D\le H_1.
 \end{gathered}
\end{equation}
The sets $E_D(K)$ have Haar measure $O(s_0)$, and their
$\delta$-boundary neighborhoods have Haar measure $O(\delta)$
for $0<\delta\le1$.

Use the separated blocks from Section~\ref{sec:prelim} with
$L=\lfloor\sqrt N\rfloor$, and let $B$ be the number of full active blocks.
For each active block $\mathcal I_b$, let $\mathcal R_b$ be the
multiset of all residuals $\pm(ma_j-na_i)$ from its constituent short
blocks, with the bounded coefficients just specified.  Write $r_N(z)$
for the total multiplicity of $z$ in these multisets.  Since the short
block length and coefficient range are fixed,
\begin{equation}\label{ac:residual-count}
 \sum_zr_N(z)\le C N.
\end{equation}
Zero residuals are included.  Put
\[
 \beta_b(x,K)=\sum_{z\in\mathcal R_b}\sum_{D=1}^{H_1}
                  \1_{E_D(K)}(zx\bmod1),\qquad
 \beta(x,K)=\sum_b\beta_b(x,K).
\]
Then $\beta_b\le CL$, $L\to\infty$, and $Ls_0\to0$.
By \eqref{ac:blockhit}, \eqref{ac:generichit}, and the ball-mixture
representation of convolution with $\Phi_{T_b}$,
\begin{align}
 h_b(x,K)&\ge(\chi_k-\varepsilon_1)Ls_0-O(N^{-A_{\mathrm{err}}}),
                                      \label{ac:averaged-class}\\
 h_b(x,K)&\ge(1-\varepsilon_1)Ls_0-Ms_0\beta_b(x,K)-O(N^{-A_{\mathrm{err}}}).
                                      \label{ac:averaged-generic}
\end{align}
Here $k=d-\operatorname{rk}(K)$, $\chi_k=1/\zeta(k)$ for $k\ge2$,
and $\chi_1=\chi(d,q_0)>0$ is from \cref{ac:positive-hit}.
For the second estimate, every exceptional short block for an averaging ball
forces \eqref{ac:incidence} at its center, so the same $\beta_b(x,K)$
bounds the exceptional count at every averaging radius.

Choose $e_N\ge0$ as in \eqref{ac:smoothed-marginal}, with
$e_N\ll\Lambda^{-1}$, and put $p_N=L(s_0+e_N)<1/2$. Define
\[
 \mathsf W_b(x,K)=\prod_{j=1}^b(1-H_j(x,K))(1+h_j(x,K)),
 \qquad \mathsf W_0=1.
\]
All these weights are nonnegative and bounded by $N^C$.
Repeating the compensated-weight recurrence from the proof of
\cref{lem:od-uniform-avoidance}, now with \cref{ac:transfer}, gives
\begin{equation}\label{ac:fourier-compensation}
 \int\mathsf W_B\,d\mu\le1+O(N^{-A_{\mathrm{err}}/2}).
\end{equation}
The buffers and final incomplete block contain
$O(N\log\Lambda/L+L+\log\Lambda)=o(N)$ indices.  Moreover,
\[
 \sum_bh_b^2\le p_N\sum_bh_b\ll NLs_0^2
             \ll\frac{(\log N)^2}{\sqrt N}=o(1).
\]
On the avoidance event,
\[
 \mathsf W_B=\prod_b(1+h_b)
 \ge\exp\Bigl(\sum_bh_b-\frac12\sum_bh_b^2\Bigr),
\]
so \eqref{ac:averaged-class} implies
\begin{equation}\label{ac:fourier-void}
 \mu\{\operatorname{int}K\cap A_N(x)=\varnothing\}
 \le e^{-(\chi_k-\varepsilon_1-o(1))Ns_0}
                                      +O(N^{-A_{\mathrm{err}}/2}).
\end{equation}
Applying Markov's inequality with \eqref{ac:averaged-generic} on
$\{\beta(x,K)\le\delta N\}$ and summing over the $N^{d+o(1)}$
approximating bodies, the probability that some empty $K$ has
$\beta(x,K)\le\delta N$ is at most
\begin{equation}\label{ac:mu-generic}
 N^{d-(1-\varepsilon_1-M\delta)(1-\varepsilon)c+o(1)}
 +O(N^{-A_{\mathrm{err}}/3}).
\end{equation}

\subsection{Exceptional targets}
\begin{lemma}\label{ac:fourier-marker}
For $|z|>\Lambda^2$, $1\le D\le H_1$, and a class-$k$ target,
\begin{equation}\label{ac:mu-marked}
 \mu\{\operatorname{int}K\cap A_N(x)=\varnothing,
                            \ zx\bmod1\in E_D(K)\}
 \le Cs_0e^{-(\chi_k-\varepsilon_1-o(1))Ns_0}
                                      +O(N^{-A_{\mathrm{err}}/3}).
\end{equation}
The implied constants and the $o(1)$ are uniform in $z,D$, and $K$.
\end{lemma}
\begin{proof}
Choose $C_{\rm sep}>6$ and put
\[
 W_z=[|z|\Lambda^{-C_{\rm sep}},|z|\Lambda^{C_{\rm sep}}].
\]
Remove every active block whose frequency range $[a_b^-,a_b^+]$ meets
$W_z$.  At most two such blocks contain frequencies outside $W_z$, so
$O(\log\Lambda)+2L=o(N)$ indices are removed.  The retained blocks
contain $N-o(N)$ indices and lie entirely below or above $W_z$.

Insert $M_z(x)=\1_{E_D(K)}(zx\bmod1)$ after the last retained block
below $W_z$.  The torus target $E_D(K)$ has Haar measure $O(s_0)$.
By \cref{lem:smoothing}, $M_z$ has a bounded finite exponential-sum
approximation with arbitrary inverse-power $L^1(\mu)$ accuracy,
polynomial Wiener norm, and spectrum in $\{\|\xi\|\le|z|N^C\}$.
Let $F$ be the compensated weight from the retained blocks preceding
the insertion, and let $a_*$ be its frequency cutoff as in
\cref{ac:slowweights} (take $F=1$ and $a_*=1$ if there is no such block).
Since $|z|\ge a_*\Lambda^2$, \eqref{ac:periodic-insertion} gives
\[
 \int FM_z\,d\mu
 =\lambda_d(E_D(K))\int F\,d\mu+O(N^{-A_{\mathrm{err}}}).
\]

Let $\mathcal B_z$ denote the retained block indices.  Keep their original
averaging scales below $W_z$.  At the first retained block above $W_z$, use
\[
 T_{b,z}=\max\{1,a_*,|z|\}\Lambda^2.
\]
Its smallest frequency divided by $T_{b,z}$ is at least $\Lambda$.
At each later retained block use the usual scale, with the preceding
frequency cutoff including $|z|$. Write $T_{b,z}$ for these scales at every
retained block, including the unchanged scales below $W_z$.  They satisfy
$a_b^-/T_{b,z}\ge\Lambda$.  After the insertion, each such scale is at
least $\Lambda^2$ times the largest frequency in the preceding retained
blocks and $|z|$.  Thus the modified-scale version of
\cref{ac:slowweights} and the hypotheses of \cref{ac:transfer} apply.
Write
\[
 \widetilde h_{b,z}=\Phi_{T_{b,z}}*H_b,\qquad
 \mathsf W_z(x,K)=\prod_{b\in\mathcal B_z}
       (1-H_b(x,K))(1+\widetilde h_{b,z}(x,K)).
\]
By \cref{ac:slowweights,ac:transfer} and the recurrence in the proof of
\cref{lem:od-uniform-avoidance}, inserting $M_z$ via
\eqref{ac:periodic-insertion} yields
\[
 \int M_z\mathsf W_z\,d\mu
 \le\lambda_d(E_D(K))+O(N^{-A_{\mathrm{err}}/2})
 \le Cs_0+O(N^{-A_{\mathrm{err}}/2}).
\]
On the avoidance event for all $N$ indices,
$\mathsf W_z=\prod_{b\in\mathcal B_z}(1+\widetilde h_{b,z})$.
There are $N-o(N)$ retained indices, so \eqref{ac:averaged-class} gives
\[
 \sum_{b\in\mathcal B_z}\widetilde h_{b,z}
 \ge(\chi_k-\varepsilon_1-o(1))Ns_0,
 \qquad
 \sum_{b\in\mathcal B_z}\widetilde h_{b,z}^2=o(1).
\]
Thus $\mathsf W_z\ge e^{(\chi_k-\varepsilon_1-o(1))Ns_0}$ on this
event, proving the assertion.
\end{proof}

For $|z|\le\Lambda^2$, including $z=0$, use the pinned inner approximating
family from \cref{ac:relative-pinned}.  The incidence gives
\[
 p(x)=\frac{zx+b}{D}\in R_1(K)\quad\text{for some }b\in\Z^d.
\]
There is a relative body $J$ of volume $(1-\varepsilon)s_0$ with
$p(x)+J\subset K$.
The projection of $\operatorname{int}J$ to $\T^d$ is injective
and volume-preserving, since the coordinate widths of $J$ are at most one.
Since $p(x)+J\subset[0,1]^d$,
avoidance of $K$ implies
\begin{equation}\label{ac:shifted-avoidance}
 \left(\left(a_n-\frac zD\right)x-\frac bD\right)\bmod1
 \notin (\operatorname{int}J)\bmod\Z^d\qquad(n\le N).
\end{equation}
Delete the initial indices until $a_n\ge2|z|+\Lambda^3$.
There are $O(\log\Lambda)=o(N)$ such indices.  The remaining positive
frequencies $\widetilde a_n=a_n-z/D$ are $q_0$-lacunary: for $z\ge0$
their ratios are at least $q$, while for $z<0$, writing $u=|z|/D$,
\[
 \frac{\widetilde a_{n+1}}{\widetilde a_n}
 \ge\frac{qa_n+u}{a_n+u}\ge\frac{q+1}{2}=q_0.
\]
The first averaging scale is admissible because the remaining frequencies
are at least $\Lambda^3$.  The constants in \cref{ac:transfer} and
\eqref{ac:fourier-void} are uniform over $q_0$-lacunary sequences with
first frequency at least one.  The common phase $-b/D$ is absorbed by
translation of the torus target, which preserves its associated frame,
scales, and rank.  Omitting $O(\log\Lambda)=o(N)$ indices changes the
hitting exponent by $o(1)$.  Thus \eqref{ac:fourier-void} applies
uniformly to \eqref{ac:shifted-avoidance}, with volume
$(1-\varepsilon)s_0$ and the same rank class.

To sum over the original targets, use the bounded-fiber assignment in
\eqref{ac:pinned-multiplicity}.  For fixed $x$, $b\pmod D$ determines $p(x)$ modulo
$\Z^d$.  Choose its representative in $[0,1)^d$; only $O(1)$ integer
lifts lie in the fixed bounded region containing all enlarged boxes,
uniformly in $x,z$.  Replacing $b$ by $b+D u$, $u\in\Z^d$, changes
$p(x)$ by $u$ but leaves \eqref{ac:shifted-avoidance} unchanged modulo
$\Z^d$.  Apply the assignment to these lifts and sum over the
$D^d$ phases.  By \eqref{ac:shifted-avoidance},
\begin{equation}\label{ac:pinned-sum}
 \sum_{K\text{ of class }k}
 \1_{\{K\text{ empty}\}}\1_{E_D(K)}(zx\bmod1)
 \le C\sum_{b\in\{0,\ldots,D-1\}^d}
             \sum_{J\in\mathscr J_{s_0,k}}
              \1_{\{\text{\eqref{ac:shifted-avoidance} holds for }J,b\}}.
\end{equation}
The relative family has $N^{k-1+o(1)}$ members.  Integrating
\eqref{ac:pinned-sum} therefore gives
\begin{equation}\label{ac:small-marker-sum}
 \sum_{K\text{ of class }k}
 \mu\{K\text{ empty},\ zx\bmod1\in E_D(K)\}
 \le N^{k-1-(\chi_k-\varepsilon_1)(1-\varepsilon)^2c+o(1)}
       +O(N^{-A_{\mathrm{err}}/4}).
\end{equation}

If $\beta(x,K)>\delta N$, then
\[
 \1_{\{K\text{ empty}\}}\1_{\{\beta(x,K)>\delta N\}}
 \le\frac1{\delta N}\sum_zr_N(z)\sum_{D=1}^{H_1}
       \1_{\{K\text{ empty}\}}\1_{E_D(K)}(zx\bmod1).
\]
Sum over the class-$k$ subfamily and integrate.  For large $z$,
\cref{ac:fourier-marker}, summed over the $N^{k+o(1)}$ class-$k$
targets, gives
\[
 N^{k-1-(\chi_k-\varepsilon_1)(1-\varepsilon)c+o(1)}
 +O(N^{-A_{\mathrm{err}}/4}).
\]
For small $z$, use \eqref{ac:small-marker-sum}.  In both cases
\eqref{ac:residual-count} gives $\sum_zr_N(z)/(\delta N)=O(1)$.
Thus the exceptional contribution
is bounded by these two displayed powers, summed over $1\le k\le d$.

The generic limiting exponent in \eqref{ac:mu-generic} is $d-c<0$.
For $k\ge2$, the exceptional limiting exponent is
$k-1-\chi_kc$, where $\chi_k=1/\zeta(k)$.  Since
\[
 \zeta(k)<1+\int_1^\infty t^{-k}\,dt=\frac{k}{k-1},
\]
we have $\chi_kd>k-1$.  For $k=1$, the exponent is
$-\chi_1c<0$.  First choose $\varepsilon,\varepsilon_1$ within these strict margins.
The short-block constants, including $M$, are then fixed.  Choose
$\delta$ sufficiently small to retain the margins, and finally
choose $A_{\mathrm{err}}$ large enough to absorb all catalogue and weight factors.  We obtain
\begin{equation}\label{ac:class-margin}
 \mu\{G_N^*>c\log N/N\}\ll N^{-\alpha}
 \qquad\text{for some }\alpha>0.
\end{equation}

\subsection{Fixed-body upper bound}
For this bound, fix $c>1$ and put $s=c\log N/N$ and
$s_0=(1-\varepsilon)s$.  Use the $O(s^{-1})$ translates from
\cref{lem:fixed-catalogue}.  The estimates \eqref{ac:averaged-class}
and \eqref{ac:averaged-generic} follow from
\cref{ac:positive-hit,ac:generic-pair-hit}, with
$\chi=\chi(d,q_0)>0$ in place of $\chi_k$.
The residual incidences \eqref{ac:incidence} and multiplicity bound
\eqref{ac:residual-count} also hold.

The generic approximating-family count is now $N^{1+o(1)}$, giving
\[
 N^{1-(1-\varepsilon_1-M\delta)(1-\varepsilon)c+o(1)}.
\]
For large residuals, the incidence factor $O(s_0)$ cancels the
$O(s_0^{-1})$ translation count.  For small residuals,
\cref{lem:fixed-catalogue} gives a relative family of bounded cardinality
and multiplicity.
The exceptional contribution is consequently at most
\[
 N^{-(\chi-\varepsilon_1)(1-\varepsilon)c+o(1)}
 +N^{-(\chi-\varepsilon_1)(1-\varepsilon)^2c+o(1)}
 +O(N^{-A_{\mathrm{err}}/4}).
\]
The limiting exponents $1-c$ and $-\chi c$ are strictly negative.
Choosing the auxiliary parameters within these margins gives
\[
 \mu\{V_N^{\mathcal C}>c\log N/N\}\ll N^{-\alpha}
 \qquad\text{for some }\alpha>0.
\]

\section{Lower bounds}\label{sec:all-convex-lower}

We extend the weighted second-moment argument of
Section~\ref{sec:one-dimensional}, following the moving-target recursion
of Peres--Yang~\cite[Section~3, Lemma~3.2]{PeresYang2026}.

\subsection{Target parameters and intersections}
For the fixed-body case, let $\Theta=\T^d$ with Haar measure $\nu$ and set
$J_\vartheta(s)=\vartheta+\operatorname{int}\mathcal C_s\pmod{\Z^d}$.
For the unrestricted family, use $J_\vartheta(s)=S_\vartheta(s)$ from
\eqref{eq:model-slabs}, with its normalized parameter measure $\nu$.

For slab targets, let $a<b$, put $R=b/a$, and let $Q$ be a cube of
side $\rho$ with $a\rho\ge1$. We use two bounds:
\begin{align}
 \mathbb P_Q(ax\bmod1\in S_\vartheta(s),bx\bmod1\in S_\varphi(s))
       &\le C(s^2+s/R),\label{ac:lower:ordinary}\\
 \mathbb P_Q(ax\bmod1\in S_\vartheta(s),bx\bmod1\in S_\varphi(s))
       &\le Cs^2/h\quad
          \text{if }\|u-u'\|_\infty\ge h.\label{ac:lower:transverse}
\end{align}
The second bound is for $d\ge2$, with $\vartheta=(u,t)$,
$\varphi=(u',t')$, and $0<h\le1$.
The first estimate is \eqref{ac:coarse} for the two slab targets, and the same
cell-counting argument applies when $1\le a\rho<4$.

For \eqref{ac:lower:transverse}, write $T=a\rho\ge1$ and change
variables to $X=ax$. Cover the cube $aQ$ by at most
$(T+2)^d\ll_d T^d$ integer unit cells $k+[0,1]^d$, and write
$X=k+\xi$ on each cell. Consider the full fast-grid cells
$\{\xi:R(k+\xi)-m\in[0,1]^d\}$, $m\in\Z^d$, before intersecting
them with the slow unit cell. These cells have disjoint interiors and volume $R^{-d}$ each.
Every cell meeting $S_\vartheta(s)$ lies in
$S_\vartheta(s)+[-R^{-1},R^{-1}]^d$. The outer-layer estimate
used in \eqref{ac:coarse} therefore bounds the number of such cells by
\[
 R^d\vol\bigl(S_\vartheta(s)+[-R^{-1},R^{-1}]^d\bigr)
 \ll_d R^{d-1}(1+Rs).
\]
On the intersection of such a cell with the slow unit cell, set
$y=R(k+\xi)-m$, so that $d\xi=R^{-d}dy$. In the $y$ coordinates,
the first slab condition has normal $(-u,1)$ and interval width $Rs$,
while the second has normal $(-u',1)$ and interval width $s$;
the additive constants depend on $k,m$ but do not affect the bound.
Choose $j_0$ with $|u_{j_0}-u'_{j_0}|\ge h$. The minor involving
coordinate $j_0$ and the last coordinate shows that the Gram
determinant of these normals is at least
$|u_{j_0}-u'_{j_0}|^2\ge h^2$.
By \cref{lem:two-slab-intersection}, their intersection in
$[0,1]^d$ has volume at most $C_dRs^2/h$. Multiplying by
$R^{-d}$ and summing over the contributing fast cells gives
$C_d(1+Rs)s^2/h$ per slow unit cell. If $R\le s^{-1}$, this is
$O_d(s^2/h)$. If $R>s^{-1}$, \eqref{ac:lower:ordinary} gives
$O_d(s^2)$, which is also $O_d(s^2/h)$ since $h\le1$.
Summing over the covering of $aQ$ and dividing by $T^d$ proves
\eqref{ac:lower:transverse}. For equal frequencies, the bound
$Cs^2/h$ follows directly from \cref{lem:two-slab-intersection}
on each unit cell.

For each fixed frequency $a_j$, the ratios $a_k/a_j$, $k>j$,
increase geometrically by \eqref{eq:lacunarity}. When $s<h$, there
are $O_q(1+\log(h/s))$ ratios at most $h/s$; the remaining terms are
bounded by a geometric series. Thus
\begin{equation}\label{ac:lower:summin}
 \sum_{k>j}\min\bigl\{s^2/h,\ s(a_j/a_k)\bigr\}
       \le C\frac{s^2}{h}\left(1+\log\frac hs\right).
\end{equation}
Combining
\eqref{ac:lower:ordinary}--\eqref{ac:lower:summin}, and including the
same-frequency intersections, shows that, when $0<s<h\le1$ and
$\|u-u'\|_\infty\ge h$, any block of $L$ indices satisfies
\begin{equation}\label{ac:lower:farblock}
 \mathbb P_Q(\text{both }S_\vartheta(s)\text{ and }S_\varphi(s)
                 \text{ are hit})
       \le C\left(L^2s^2+\frac{Ls^2}{h}\left(1+\log\frac hs\right)\right).
\end{equation}

For each $y_0\in\T^d$ and $\vartheta=(u,t)$,
\begin{equation}\label{ac:lower:parameter}
 \nu\{\varphi=(u',t'):\|u-u'\|_\infty<h,\
                  y_0\in S_\varphi(s)\}
       \le Cs h^{d-1}.
\end{equation}
For each $u'$, the allowable intercepts $t'$ lie in an interval of
length at most $s$; the allowable slopes have
volume $O(h^{d-1})$. In dimension one the corresponding bound is
$Cs$, with the slope condition omitted.

For fixed-body translates, Fubini gives
\begin{equation}\label{eq:translation-incidence}
 \nu\{\vartheta:y\in J_\vartheta(s)\}=s
 \qquad(y\in\T^d).
\end{equation}
Thus any $L$ points hit targets with total parameter measure at most $Ls$.

Put $s=c\log N/N$ and $L=\lfloor N^\xi\rfloor$.
For fixed-body targets, or slabs when $d=1$, fix $0<c<1$ and take
\[
 \xi=\frac{1-c}{2},
\]
declare every parameter pair near, and put $\omega_N=1$.
For slabs with $d\ge2$, fix $0<c<d$ and take
\begin{equation}\label{ac:lower:choices}
 \xi=\min\left\{\frac12,\frac{d-c}{d+1}\right\},\qquad
 h=N^{-(1-\xi)},\qquad \omega_N=h^{d-1}.
\end{equation}
Declare $(u,t)$ and $(u',t')$ near when $\|u-u'\|_\infty<h$.
These choices give $Ls=o(1)$ and, for slabs with $d\ge2$, $s/h=o(1)$.
Write $\mathcal N(\vartheta)$ for the set of near parameters, and put
$H_{b,\vartheta}(x)=\1_{\{a_nx\bmod1\in J_\vartheta(s)
\text{ for some }n\in\mathcal I_b\}}$.
By \eqref{ac:lower:parameter} and \eqref{eq:translation-incidence},
\begin{equation}\label{ac:actual-near}
 \int_{\mathcal N(\vartheta)}H_{b,\varphi}(x)\,d\nu(\varphi)
 \le CLs\omega_N.
\end{equation}
For far slab pairs, discard from the ball mixture the radii
$r<C/a_b^-$.  Their total mixture mass is
\[
 O\!\left((T_b/a_b^-)^d\right)=O(\Lambda^{-d})=O(N^{-A})
\]
for every fixed $A>0$.  Each remaining ball of radius $r$ lies in an axis-parallel cube of
side $\rho=2r$ whose volume is $O_d(1)$ times the ball volume.
Since $a_b^-\rho\ge1$, normalized ball averages are bounded by a
dimensional constant times the corresponding cube averages in
\eqref{ac:lower:farblock}. Hence
\begin{equation}\label{ac:averaged-farpair}
 \Phi_{T_b}*(H_{b,\vartheta}H_{b,\varphi})
 \le C\left(L^2s^2+\frac{Ls^2}{h}\left(1+\log\frac hs\right)\right)
       +O(N^{-A}).
\end{equation}
Set $v_N=(Ls^2/h)(1+\log(h/s))$ for slabs with $d\ge2$, and
$v_N=0$ otherwise.

\subsection{The weighted second moment}
\begin{proposition}\label{ac:lower:main}
Assume \eqref{eq:lacunarity} and \eqref{eq:fourier-decay}.
For each fixed $c$ in the indicated range,
\[
\begin{aligned}
 \mu\{V_N^{\mathcal C}<c\log N/N\}
   &=O\bigl(N^{-(1-c)/2}(\log N)^7\bigr) &&(0<c<1),\\
 \mu\{G_N^*<c\log N/N\}
   &=O\bigl(N^{-\alpha_d(c)}(\log N)^7\bigr) &&(0<c<d),
\end{aligned}
\]
where
\[
 \alpha_d(c)=\min\left\{\frac12,\frac{d-c}{d+1}\right\}.
\]
The implied constants may also depend on $c$.
\end{proposition}
\begin{proof}
Fix either target family, with $s,L,\omega_N,v_N$ chosen above
(and $h$ in the far-pair case).
Use the active blocks from Section~\ref{sec:prelim}, with common length
$L$, buffers of $g=\lceil C_{\rm gap}\log\Lambda\rceil$ omitted
indices, and $T_b=a_{b-1}^+\Lambda^2$.  Let $B$ be the number of full active blocks.
All expectations in this proof are with respect to $\mu$.
Choose the accuracy exponent $A$ so that $O(N^{-A})$ errors remain
negligible after at most $N$ iterations and division by $N^{-2c+o(1)}$.
Take $e_N\ge0$ from \eqref{ac:smoothed-marginal}, so that
$e_N=O(N^{-A})$, and put
\[
 h_{b,\vartheta}=\Phi_{T_b}*H_{b,\vartheta},\qquad
 p=L(s+e_N),\qquad \theta=1-p.
\]
Thus $0\le h_{b,\vartheta}\le p=o(1)$. As in the one-dimensional
argument, define the damped compensated weight for the present blocks
and target family:
\[
 \mathsf U_b(\vartheta)=\theta^b\prod_{j=1}^b
   (1-H_{j,\vartheta})(1+h_{j,\vartheta}),\qquad
 \mathsf Z_b=\int_\Theta \mathsf U_b(\vartheta)\,d\nu(\vartheta),
\]
with empty products equal to one.  Thus $\mathsf Z_b$ is the total weighted
mass of targets surviving the first $b$ blocks. Equivalently,
$\mathsf U_0=1$ and
\[
 \mathsf U_b=\theta\mathsf U_{b-1}(1-H_{b,\vartheta})(1+h_{b,\vartheta}).
\]
Since $\theta(1+h_{b,\vartheta})\le(1-p)(1+p)\le1$, we have
$0\le \mathsf U_b\le\mathsf U_{b-1}\le1$.  In particular,
$\mathsf Z_B>0$ implies that some target avoids all active blocks.

Repeating the calculation leading to \eqref{od:lower-first}, with
\cref{ac:transfer} in place of the one-dimensional transfer, gives,
uniformly in $b\le B$ and $\vartheta$,
\begin{equation}\label{ac:lower:firstmoment}
 \mathbb E \mathsf U_b(\vartheta)
 =\theta^b(1+O(Bp^2))+O(N^{-A/2}).
\end{equation}
The number of omitted indices is
\[
 m_N=N-BL=O(N\log\Lambda/L+L+\log\Lambda).
\]
In particular,
\[
 m_Ns=O(N^{-\xi}\ell^7+N^{\xi-1}\ell),\qquad
 Bp^2=O(N^{\xi-1}\ell^2).
\]
Since $Bp=Ns-m_Ns+O(Ne_N)$,
\[
 \theta^B=\exp\{-Ns+O(m_Ns+Bp^2+Ne_N)\}.
\]
Using $\xi\le1/2$, we obtain
\begin{equation}\label{ac:mu-mean}
 \mathbb E \mathsf Z_B=\theta^B(1+O(Bp^2))+O(N^{-A/2})
             =e^{-Ns}(1+O(N^{-\xi}\ell^7)).
\end{equation}
In particular, $\theta^{-B}=N^c(1+O(N^{-\xi}\ell^7))$.

For the second moment, set
$F=\mathsf U_{b-1}(\vartheta)\mathsf U_{b-1}(\varphi)$. The pointwise
expansion from \eqref{od:lower-second-step} and
\cref{ac:transfer} give
\begin{align*}
 \mathbb E[\mathsf U_b(\vartheta)\mathsf U_b(\varphi)]
 \le\theta^2\bigl((1+Cp^2)\mathbb EF
 +(1+Cp)\mathbb E[FH_{b,\vartheta}H_{b,\varphi}]\bigr)+O(N^{-A}).
\end{align*}
For far pairs, the two-hit transfer and \eqref{ac:averaged-farpair}
bound the last expectation by
$C(L^2s^2+v_N)\mathbb EF+O(N^{-A})$.
For near pairs, $\mathsf U_{b-1}(\varphi)\le1$ and
\eqref{ac:actual-near} give
\[
 \int_{\mathcal N(\vartheta)}F H_{b,\vartheta}H_{b,\varphi}\,d\nu(\varphi)
 \le CLs\omega_N \mathsf U_{b-1}(\vartheta)H_{b,\vartheta}.
\]
The one-hit transfer bounds its expectation by
$CL^2s^2\omega_N\mathbb E \mathsf U_{b-1}(\vartheta)+O(N^{-A})$.
Set $M_b=\mathbb E\mathsf Z_b^2$ and $\epsilon_N=L^2s^2+v_N$.
By \eqref{ac:lower:firstmoment},
\[
 M_b\le\theta^2(1+C\epsilon_N)M_{b-1}
       +CL^2s^2\omega_N\theta^{b-1}+O(N^{-A/2}).
\]
Normalize by $\theta^{2b}$ and iterate.  Since
$B\epsilon_N=o(1)$ and
\[
 \sum_{b=1}^B\theta^{-b-1}
 =\frac{\theta^{-1}(\theta^{-B}-1)}{p}
 \le C\theta^{-B}/p,
\]
this gives
\[
 \frac{M_B}{\theta^{2B}}
 \le 1+O\bigl(B\epsilon_N+Ls\omega_N\theta^{-B}\bigr)+O(N^{-A/4}).
\]
Subtracting the squared mean in its $\theta^B$ form in
\eqref{ac:mu-mean}, and using $Bp^2\ll B\epsilon_N$, yields
\begin{equation}\label{ac:mu-variance}
 \frac{\operatorname{Var}\mathsf Z_B}{(\mathbb E\mathsf Z_B)^2}
 \le C\bigl(NLs^2+(N/L)v_N+Ls\omega_N e^{Ns}\bigr)+O(N^{-A/4}).
\end{equation}
For slabs with $d\ge2$, \eqref{ac:lower:choices} gives
\[
 \begin{aligned}
 NLs^2&\ll N^{\xi-1}\ell^2,\\
 (N/L)v_N&\ll N^{-\xi}\ell^3,\\
 Ls\omega_N e^{Ns}&\ll N^{c-d+d\xi}\ell.
 \end{aligned}
\]
Since $\xi\le1/2$ and $(d+1)\xi\le d-c$, each term is
$O(N^{-\xi}\ell^3)$. When all pairs are near, the two nonzero
terms are $O(N^{\xi-1}\ell^2)$ and $O(N^{-\xi}\ell)$.
Thus the relative variance is $O(N^{-\xi}\ell^3)$ in both cases.
These bounds also verify $B\epsilon_N=o(1)$.

We now restore the omitted indices. Recall that $X_n(x)=a_nx\bmod1$.
We claim that every omitted index $n$ satisfies
\begin{equation}\label{ac:one-omitted}
 \mathbb E\int_\Theta \mathsf U_B(\vartheta)
              \1_{\{X_n\in J_\vartheta(s)\}}\,d\nu(\vartheta)
 \le Cs\theta^B+O(N^{-A/4}).
\end{equation}
As in \eqref{od:one-omitted}, removing block factors increases
$\mathsf U_B$.  If $n$ lies between blocks $b$ and $b+1$, remove both
factors; a full block of length $L$ separates $a_n$ from the retained
frequencies on either side.  Since $L\asymp N^\xi$, $q^L$ dominates
every fixed power of $\Lambda$.  Thus \eqref{ac:periodic-insertion} gives
\begin{equation}\label{ac:omitted-separated}
 \mathbb E\bigl[\mathsf U_{b-1}(\vartheta)
                   \1_{\{X_n\in J_\vartheta(s)\}}\bigr]
 \le s\,\mathbb E\mathsf U_{b-1}(\vartheta)+O(N^{-A}).
\end{equation}
For $b=1$, the preceding weight is $1$ and $a_n\ge a_1q^L$, so the
same estimate applies.  Iterating the inserted-hit version of
\cref{ac:transfer} over the retained later blocks gives
$Cs\theta^{B-2}+O(N^{-A/4})$.
After the final active block, delete its factor and obtain
$Cs\theta^{B-1}+O(N^{-A/4})$.

For an index before the first active block, delete the first block factor
and start the later recursion from the inserted hit.  Its initial expectation,
averaged over the target parameter, is at most $Cs$, because pointwise
in $x$,
\begin{equation}\label{ac:omitted-parameter}
 \nu\{\vartheta:X_n(x)\in J_\vartheta(s)\}\le Cs.
\end{equation}
This is \eqref{eq:translation-incidence} for fixed-body translates;
for slabs, fixing the slope leaves intercepts in an interval of length $s$.
With $\delta=N^{-B_0}$ and $B_0$ sufficiently large,
\eqref{eq:early-hit-smoothing} gives the necessary polynomial
approximation in $L^2(\mu\otimes\nu)$, even when $a_n$ is small.
After deleting the first block, the next retained block is separated
from $a_n$ by at least $L$ indices, so its frequency dominates the
polynomial cutoff by every fixed power of $\Lambda$. The
parameter-integrated form of \cref{ac:transfer} therefore applies,
and iteration gives $Cs\theta^{B-1}+O(N^{-A/4})$.
Since $\theta^{-2}=O(1)$, the three cases prove \eqref{ac:one-omitted}.

Let $\mathcal O_N$ be the set of omitted indices, and put
\[
 \mathsf S_B(x)=\int_\Theta \mathsf U_B(\vartheta)
       \1_{\{X_n(x)\in J_\vartheta(s)\text{ for some }n\in\mathcal O_N\}}
                  \,d\nu(\vartheta).
\]
The union bound, \eqref{ac:one-omitted}, and \eqref{ac:mu-mean} give
\begin{equation}\label{ac:mu-gaps}
 \frac{\mathbb E\mathsf S_B}{\mathbb E\mathsf Z_B}
 \le Cm_Ns+O(N^{-A/8})=O(N^{-\xi}\ell^7).
\end{equation}
By Chebyshev and Markov, except on a set of measure
$O(N^{-\xi}\ell^7)$,
\[
 \mathsf Z_B\ge\tfrac12\mathbb E\mathsf Z_B,
 \qquad \mathsf S_B\le\tfrac14\mathbb E\mathsf Z_B.
\]
Hence $\mathsf Z_B-\mathsf S_B>0$, so some target of volume $s$
avoids all $N$ points. The choices of $\xi$ give the two stated bounds;
when $d=1$, $\alpha_1(c)=(1-c)/2$.
\end{proof}

\begin{proof}[Completion of \cref{thm:fixed-shape,thm:all-convex}]
Absorbing the logarithmic factors into a smaller positive exponent,
Section~\ref{sec:all-convex-upper} and \cref{ac:lower:main} imply that
for each $\varepsilon>0$ there are constants $\alpha,C>0$ such that
\begin{equation}\label{eq:main-tails}
 \begin{aligned}
 &\mu\left\{\left|\frac{NV_N^{\mathcal C}}{\log N}-1\right|>\varepsilon\right\}+\mu\left\{\left|\frac{NG_N^*}{\log N}-d\right|>\varepsilon\right\}
 \le C N^{-\alpha}\qquad(N\ge2).
 \end{aligned}
\end{equation}
Apply the geometric-subsequence and monotonicity argument from
Section~\ref{sec:one-dimensional} to \eqref{eq:main-tails}, with
$N_j=\lfloor r^j\rfloor$ for $r>1$.  Let $r\downarrow1$ and
$\varepsilon\downarrow0$ through countable sequences.  This proves both limits.
\end{proof}

\section{A sparse counterexample}\label{sec:general-sparse-counterexample}
\begin{proof}[Proof of \cref{thm:sparse}]
We construct $(a_n)$ in finite arithmetic blocks. Fix $j\ge1$, suppose
that the preceding blocks have lengths $K_1,\ldots,K_{j-1}$, and put
$C_j=\sum_{i<j}K_i$.

By Khintchine's theorem (see Cassels~\cite[Chapter~VII]{Cassels1957}),
for Lebesgue-almost every $\theta\in[0,1)$ there are infinitely many reduced
fractions $p/q$ such that
\[
 \left|\theta-\frac pq\right|<\frac1{8jq^2\log q}.
\]
With $j$ and $C_j$ fixed, we may take $q>j(C_j+1)$ arbitrarily large.
Since $(p,q)=1$, the points $rp/q\bmod1$, $1\le r\le q$, form the
uniform $q$-grid. The first $q$ multiples of $\theta$ lie within
$q|\theta-p/q|$ of this grid and therefore have maximal gap less than
\[
 \frac1q+2q\left|\theta-\frac pq\right|<\frac{1+1/j}{q}.
\]

Set $m=\lfloor4jq\log q\rfloor$. For $1\le r\le m$,
\[
 r\left|\theta-\frac pq\right|<\frac1{2q}.
\]
All displacements $r(\theta-p/q)$ have the same sign. Hence the
first $m$ multiples lie in at most one half of each of the $q$ grid
intervals, leaving $q$ disjoint open arcs of length $1/(2q)$. Since $j,C_j$
are fixed and $m\sim4jq\log q$, we have
$\log(C_j+m)\sim\log q$. Hence, for sufficiently large $q$,
\[
 \frac1{2q}>\frac{j\log(C_j+m)}{C_j+m}.
\]

By continuity from below, we can choose an integer $K_j\ge2$ such
that, outside a set of $\theta$ of Lebesgue measure at most $2^{-j}$,
there are $q<m\le K_j$ with $q>j(C_j+1)$ satisfying both the
short-orbit maximal-gap bound and the long-orbit empty-arc bound above.
Since $b_k\to\infty$ and $b_{k+1}/b_k\to1$, choose an integer $d_j$
so large that, when $j>1$, some endpoint $b_\ell$ satisfies
\[
 K_{j-1}d_{j-1}<b_\ell<d_j,
\]
and
\[
 \frac{b_{k+1}}{b_k}<1+\frac1{K_j}
 \qquad\text{whenever }b_{k+1}\ge d_j.
\]
Put $\mathcal A_j=\{d_j,2d_j,\ldots,K_jd_j\}$ and let $(a_n)$ be the
increasing enumeration of $\bigcup_{j\ge1}\mathcal A_j$.
For $1\le r<s\le K_j$,
\[
 \frac{sd_j}{rd_j}\ge\frac{K_j}{K_j-1}>1+\frac1{K_j}.
\]
If both elements belonged to $(b_k,b_{k+1}]$, then
$b_{k+1}/b_k>s/r>1+1/K_j$, contrary to the choice of $d_j$.
For $j>1$, all earlier terms lie below $b_\ell$, while all terms
of $\mathcal A_j$ lie above it. Thus no interval meets two different
blocks either, proving the required sparsity.

Multiplication by $d_j$ modulo $1$ preserves Lebesgue measure. Since
$\sum_j2^{-j}<\infty$, Borel--Cantelli implies that for
Lebesgue-almost every $x$ and all sufficiently large $j$, there are
$q_j<m_j\le K_j$ with the above properties for $\theta=d_jx\bmod1$.
In particular, $q_j>j(C_j+1)$.

At $N_j=C_j+q_j$, the earlier $C_j$ points can only decrease the
maximal gap. Consequently,
\[
 1\le N_jG_{N_j}(x)
 <\left(1+\frac{C_j}{q_j}\right)\left(1+\frac1j\right)
 <\left(1+\frac1j\right)^2.
\]
At $M_j=C_j+m_j$, the current block leaves $q_j>C_j$ disjoint open
arcs, each longer than $j\log M_j/M_j$. The earlier $C_j$ points
can meet at most $C_j$ of these arcs, so one remains empty. Hence
\[
 \frac{M_jG_{M_j}(x)}{\log M_j}>j.
\]
As $j\to\infty$, these two inequalities prove \cref{thm:sparse}.
\end{proof}

\smallskip
\noindent\textbf{Acknowledgements.} We thank Yitwah Cheung for asking about higher dimensions. Y. Peres was supported by the National Natural Science Foundation of China, grant RFIS-W2531011.

\smallskip
\noindent\textbf{Declaration of AI assistance.} During the preparation of this manuscript, the authors used AI for language editing, literature searches, and assistance in exploring and checking mathematical arguments. The authors provided the mathematical ideas and directed the development of the arguments; all results and proofs were independently verified by the authors, who take full responsibility for the content of the article.

\medskip
\noindent E-mail addresses:
\href{mailto:yuval@yuvalperes.com}{\texttt{yuval@yuvalperes.com}} (Y. Peres),
\href{mailto:bhyang@simis.cn}{\texttt{bhyang@simis.cn}} (B. Yang).
\end{document}